\documentclass[10pt]{amsart}
\usepackage{amsmath}
\usepackage{amsfonts}
\usepackage{amssymb}
\usepackage{amsthm}
\usepackage{url}
\usepackage{tikz-cd}
\usepackage{dsfont}
\usepackage{graphicx}
\usepackage{caption}
\usepackage{subcaption}
\usepackage{comment}
\usepackage{stmaryrd}
\usepackage{hyperref}
\usepackage{todonotes}
\usepackage{color}
\usepackage{enumerate}

\numberwithin{equation}{section}

\newtheorem{proposition}{Proposition}[section]
\newtheorem{lemma}[proposition]{Lemma}
\newtheorem{theorem}[proposition]{Theorem}

\theoremstyle{definition}
\newtheorem{remark}[proposition]{Remark}
\newtheorem{definition}[proposition]{Definition}

\DeclareMathOperator{\Bl}{Bl}
\DeclareMathOperator{\End}{End}
\DeclareMathOperator{\Id}{Id}

\DeclareMathOperator{\Aut}{Aut}
\DeclareMathOperator{\Ric}{Ric}

\newcommand{\ddb}{i\partial \bar\partial}
\newcommand{\dbar}{\bar\partial}

\newcommand{\scH}{\mathcal{H}}
\newcommand{\scL}{\mathcal{L}}

\newcommand{\scV}{\mathcal{V}}
\newcommand{\scX}{\mathcal{X}}

\title[Extremal metrics on destabilising test configurations]{Extremal metrics on the total space of destabilising test configurations}

\author[Lars Martin Sektnan and Cristiano Spotti]{Lars Martin Sektnan and Cristiano Spotti}
\address{Lars Martin Sektnan, Department of Mathematical Sciences, University of Gothenburg, 412 96 Gothenburg, Sweden and Institut for Matematik, Aarhus University, 8000, Aarhus C, Denmark}
\email{sektnan@chalmers.se}

\address{Cristiano Spotti, Institut for Matematik, Aarhus University, 8000, Aarhus C, Denmark}
\email{c.spotti@math.au.dk}

\begin{document}

\begin{abstract} 
We construct extremal metrics on the total space of certain destabilising test configurations for strictly semistable K\"ahler manifolds. This produces infinitely many new examples of manifolds admitting extremal K\"ahler metrics. It also shows for such metrics a new phenomenon of jumping of the complex structure along fibres.
\end{abstract}

\maketitle

\section{Introduction}

A central theme in complex geometry is the search for canonical metrics, such as K\"ahler--Einstein metrics, constant scalar curvature K\"ahler (cscK) metrics, or, more generally, extremal K\"ahler metrics. Such metrics are solutions to highly complicated non-linear PDEs, and their existence question is very involved. There are examples of K\"ahler manifolds that admit extremal metrics in all, some, or none of its K\"ahler classes. The main question one would like to answer is if there is an extremal metric in a given K\"ahler class on a K\"ahler manifold.

A central conjecture in the field is the Yau--Tian--Donaldson (YTD) conjecture, which links this purely differential geometric PDE question to algebraic geometry (\cite{yau93, tian97, donaldson02}). More precisely, this says that the existence of a cscK metric should be equivalent to a notion of stability called K-(poly)stability. It was extended to the extremal setting by Sz\'ekelyhidi (\cite{Szethesis}), where the relevant algebraic notion is relative K-stability. Despite a huge amount of work over many years, the conjecture remains open apart from a few cases, such as the K\"ahler-Einstein case (\cite{chendonaldsonsun15i,chendonaldsonsun15ii,chendonaldsonsun15iii}) and the toric case (\cite{donaldson09,chencheng18}). 

Producing extremal metrics is in general very hard. Even if the YTD conjecture is proved, this will continue to be the case, as the stability criterion is also difficult to check with the present technology, at least outside the case of Fano varieties. One main avenue for actually producing extremal metrics has been to directly solve the equation, working with an ansatz that simplifies the PDE (see e.g. \cite{calabi82,hwangsinger02,acgtf04}). This requires one to work under certain symmetry assumptions. Another method is to construct new extremal metrics from old ones, via perturbative techniques. There have been many different such constructions, for example on blowups via gluing (\cite{arezzopacard06, arezzopacard09, arezzopacardsinger11, szekelyhidi12}), or on the total spaces of fibrations, working in so-called adiabatic K\"ahler classes, that make the base direction large compared to the fibres (\cite{hong98,fine04,bronnle15,dervansektnan19a,dervansektnan20}). The current work falls into the latter category. 

Previous fibration constructions have all been focused on constructing cscK or extremal metrics when the fibration has cscK fibres. The main reason for this assumption is that an asymptotic expansion of the scalar curvature with respect to metrics of the form $\omega + k \omega_B$, where $\omega$ is a relatively K\"ahler metric on the total space, and $\omega_B$ is pulled back from the base, yields the scalar curvature of the induced metric on the fibres as the leading order term. Thus, if working with a fixed relatively K\"ahler metric initially, one needs a relatively cscK metric in order to have constant scalar curvature to leading order. 

In \cite{fine04}, it was assumed that the total space and the fibres have no automorphisms. In fact, when there is a discrepancy between the automorphism group of the fibres and the automorphisms of the total space of the fibration, the existence question becomes much more subtle. For example, on projective bundles, any hermitian metric on the bundle gives a fibrewise Fubini-Study metric on the projectivised bundle. Ideas going back to Hong in \cite{hong98} say that one should use a Hermite--Einstein metric as a good choice of relatively cscK metric, in this setting. This issue was considered by many authors before Dervan and the first named author in \cite{dervansektnan20} introduced an equation for the fibrewise cscK metric, called the optimal symplectic connection (OSC) equation, that picks out a canonical choice of fibrewise cscK metric, when the fibres have automorphism.

In the present work, we relax the condition that every fibre admits a cscK metric, and consider a special type of fibration where the general fibre does not admit such metrics. In fact, in our construction, all but one fibre does not admit a cscK metric. We will construct extremal metrics on the total space of this fibration, and this shows that relative stability of the fibres is not a necessary condition (although relative \emph{semistability} is, see \cite{dervansektnan19b}).

Suppose that $(X,L)$ is an \emph{analytically K-semistable manifold}, meaning that there exists a degeneration of $(X,L)$ to some $(X_0, L_0)$ which admits a cscK metric. That is, we have a smooth test configuration
$$ 
\pi : \scX \to \mathbb{P}^1
$$ 
and a $\mathbb{Q}$-line bundle $\mathcal{L}$ on $\scX$ whose restriction to non-zero fibres equals $L$ and whose central fibre is $(X_0, L_0)$, where $X_0$ admits a cscK metric in $c_1 (L_0)$. Note that analytic K-semistability implies K-semistability (\cite{donaldson05}). 

We will let $ \scL_k = \scL + \pi^* \left( \mathcal{O} (k) \right)$. Under certain conditions, we produce an extremal metric on $\scX$ in $c_1 (\scL_k)$ for $k \gg 0$.
\begin{theorem}
\label{thm:main}
Suppose $(X,L)$ is a smooth analytically K-semistable manifold such that the automorphism group of $(X,L)$ has a $\mathbb{C}^*$ discrepancy against the automorphism group of the cscK central fibre $(\scX_0, \mathcal{L}_0)$. Then there exists an extremal metric on $\scX$ in $c_1 \left( \scL + \pi^* \left( \mathcal{O} (k) \right) \right)$ for all $k \gg 0$.
\end{theorem}
The $\mathbb{C}^*$ discrepancy condition is a technical condition, see Definition \ref{def:automassumption}. In the simplest setting where the automorphism group of $(X,L)$ is discrete, this simply means that the central fibre has a $\mathbb{C}^*$ as its maximal torus, which is the minimum it can be.

If $(X,L)$ itself is cscK in the above, so that the test configuration is a product test configuration, the result follows in a straightforward way e.g. from \cite{dervansektnan20}. Thus the interesting case is when $(X,L)$ is strictly semistable, and so does not admit a cscK metric. The key issue to overcome in this situation is that working with a fixed relatively K\"ahler metric on the total space will never yield a good approximate solution to the extremal equation, as the leading order term in the expansion of the scalar curvature will not be a holomorphy potential. Thus one has to work with a \emph{sequence} of relatively K\"ahler metrics, and constructing such a sequence is a key step in the argument.

The fact that one should be able to choose such a sequence for a general K-semistable manifold follows from a conjecture of Donaldson (\cite{donaldson05}), which states that for a polarised manifold $(Y, H)$,
$$
\inf_{\omega \in c_1 (H) } \| S(\omega) - \hat S \| = \sup_{\mathcal{Y} \textnormal{ t.c.}} \frac{- DF (\mathcal{Y}) }{\| \mathcal{Y} \|},
$$
where the supremum on the right hand side goes over all test configurations for $(Y,H)$ and $\| \cdot \|$ is a norm on test configurations. In particular, if $(Y, H)$ is strictly semistable, this supremum should be zero, and so we should be able to find metrics that are \emph{arbitrarily close to being cscK}, even if no actual cscK metric exists. Under our assumption that $(X,L)$ is \emph{analytically} K-semistable, it is automatic that we can produce such a sequence near the central fibre. Thus the main step for us in this part of the argument is to extend this relatively K\"ahler metric defined near the central fibre, to the full test configuration. 

Our assumption on the automorphism group can be thought of as analogous to the assumptions of Fine mentioned above. While the total space of the test configuration always has to have non-trivial automorphism group, our assumption says that there is no discrepancy between the automorphism group of the central fibre and that of the total space of the test configuration. In particular, we see no issue related to the \emph{choice} of cscK metric on the central fibre, and so see nothing like the OSC equation coming in. It should be possible to relax these conditions, introducing an OSC like equation in the semistable setting, and, indeed, such a theory is currently under development by Annamaria Ortu (\cite{ortu21}). See Remark \ref{rem:outlook} for further details.

The result also has some consequences and applications. A natural question to ask is about what happens from a metric viewpoint to the extremal metrics when $k\rightarrow \infty$. Our construction shows that a phenomenon of \emph{jumping of the complex structures} within the fibres is occurring. More precisely:

\begin{proposition}
	For \emph{any}  $p\in \mathcal{X}$ the pointed Gromov-Hausdorff limits based at $p$ is isometric to $X_0 \times \mathbb{R}^2$, where $X_0$ is the cscK central fibre. 
\end{proposition} 

Importantly, we can use our result to actually produce many new manifolds admitting extremal metrics. Often in the construction of extremal metrics on fibrations, it is difficult to actually generate examples. This is because the constructions requires one to solve a complicated PDE on the base of the fibration, at least when the fibres are not all isomorphic. In our case, the fibres are not all isomorphic, as there is a jump of biholomorphism class between the central fibre and the general fibre. However, the base equation just reduces to the equation for the Fubini-Study metric on $\mathbb{P}^1$, and so we see no issues coming from the base. 

To obtain examples, we therefore have to find explicit families of K-semistable manifolds degenerating to a cscK manifold. In general, this is a hard question, but a plethora of examples are known to exist, due to massive recent progress on the study of K-stability for Fano manifolds, and in particular for threefolds (see \cite{fanothreefolds} and the references therein). This allows us to produce many new manifolds admitting extremal metrics, see Theorem \ref{prop:examples}.

\subsection*{Outline} In Section \ref{sec:background} we recall some basic theory on cscK metrics relevant to our problem. In Section \ref{sec:metricsextension}, we provide a starting point for the construction, by defining a good relatively K\"ahler metric on the test configuration. Our main result, Theorem \ref{thm:main}, is proved in Section \ref{sec:construction}. In Section \ref{sec:GH} we provide applications to Gromov--Hausdorff limits of K-semistable manifolds, and finally in Section \ref{sec:examples}, we provide multiple examples of manifolds to which our construction applies.

\subsection*{Acknowledgements} We thank Ruadha\'i Dervan and Annamaria Ortu for enlightening discussions and comments related to this work. LMS is funded by a Marie Sk\l{}odowska-Curie Individual Fellowship, funded from the European Union's Horizon2020 research and innovation programme under grant agreement No 101028041. CS is funded by Villum Fonden Grant 0019098.

\section{Background on cscK metrics}
\label{sec:background}

In this section we recall some general theory regarding the cscK/extremal equation, its linearisation, and the deformation theory for cscK metrics. A good reference for most of the material in this section is the book of Sz\'ekelyhidi (\cite{szekelyhidi14book}).

\subsection{Constant scalar curvature and extremal metrics}

Let $X$ be a compact K\"ahler manifold of dimension $n$, with $\omega$ the K\"ahler form of a K\"ahler metric on $X$. Any other K\"ahler form on $X$ in the same class as $\omega$ can be described via a function, called a K\"ahler potential. More precisely, if we let $\omega_{\phi}$ for a real-valued function $\phi$ denote
$$
\omega_{\phi} = \omega + \ddb \phi ,
$$
then the space of K\"ahler potentials $\mathcal{K}$ parametrising the K\"ahler forms in the class $[\omega]$ is given by
$$
\mathcal{K} = \{ \phi : \omega_{\phi} > 0 \}.
$$

Associated to $\omega_{\phi}$ is the Ricci form 
$$
\rho_{\phi} = - \ddb \log (\omega_{\phi}^n ).
$$  
The corresponding bilinear form is the Ricci curvature of the metric associated to $\omega_{\phi}$. The scalar curvature of this metric is given by contracting with respect to the metric,
$$
S(\omega_{\phi}) = \Lambda_{\omega_{\phi}} ( \rho_{\phi}).
$$
We say $\omega_{\phi}$ has \emph{constant scalar curvature} if there exists a constant $c$ such that
$$
S(\omega_{\phi}) = c.
$$
The constant $c$ is predetermined by the topology of $X$ and the class $[\omega]$.

A generalisation of constant scalar curvature metrics is \emph{extremal} K\"ahler metrics, introduced by Calabi (\cite{calabi82}). These are defined as the critical points of the functional 
$$
\phi \mapsto \int_X S(\omega_{\phi}) \omega_{\phi}^n.
$$
Since $X$ is compact, this is equivalent to 
$$
\mathcal{D} ( S(\omega_{\phi} ) ) = 0,
$$
where $\mathcal{D} (f) = \dbar ( \nabla^{1,0}_{\omega_{\phi}} (f) ).$ This says that the gradient of the scalar curvature is a holomorphic vector field. Such functions are referred to as \emph{holomorphy potentials}.

The holomorphic vector fields that admit a \emph{complex-valued} holomorphy potential are precisely given by the holomorphic vector fields with a zero somewhere (\cite[Theorem 1]{lebrunsimanca94}). On the other hand, the space of holomorphic vector fields with a real holomorphy potential may depend on the metric. However, if the metric is invariant under the action of a fixed maximal torus, then the space of real holomorphic vector field does not depend on the metric, and the holomorphy potentials change in a predictable way. If we let $\overline{\mathfrak{h}}$ denote the space of real holomorphy potentials with respect to $\omega$, then a holomorphy potential for $\nabla^{1,0} h$ with respect to $\omega_{\phi}$ is given by
$$
h + \frac{1}{2} \langle \nabla h, \nabla \phi \rangle.
$$
 Thus, if our background metric is torus invariant, this means that to solve the extremal equation we seek a torus invariant function $\phi$ and a holomorphy potential $h$ with respect to $\omega$ such that
$$
S(\omega_{\phi}) = h + \frac{1}{2}  \langle \nabla h , \nabla \phi \rangle.
$$

\subsection{The linearisation of the scalar curvature operator}

The linearisation of the scalar curvature operator 
$$
\phi \mapsto S(\omega_{\phi})
$$
will be key in our construction. This operator is closely linked to the \emph{Lichnerowicz operator}, defined as follows. 

Let $\mathcal{D} = \mathcal{D}_{\omega} : C^{\infty}(X) \to \Gamma ( T^{1,0} X \otimes \Lambda^{0,1} X )$ be the operator 
$$
\mathcal{D} (f) = \dbar ( \nabla^{1,0}_{\omega} (f) ),
$$
as above. Let $\mathcal{D}_{\omega}^* : \Gamma ( T^{1,0} X \otimes \Lambda^{0,1} X ) \to C^{\infty}(X)$ be its formal adjoint, with respect to $\omega$. The Lichnerowicz operator is then the fourth order  self-adjoint operator 
$$
L_{\omega} : C^{\infty}(X) \to C^{\infty}(X)
$$
defined as 
$$
L_{\omega} = \mathcal{D}_{\omega}^*\mathcal{D}_{\omega}.
$$
This operator admits an expansion
$$
L_{\omega} (f) = \Delta^2 (f) + \langle \Ric (\omega), \ddb (f) \rangle_{\omega} + \frac{1}{2} \langle \nabla S (\omega) , \nabla f \rangle_{\omega},
$$
which in particular shows that it is elliptic. Note also that the kernel, and hence the cokernel, of $L_{\omega}$ consists precisely of the space $\overline{\mathfrak{h}}$ of real holomorphy potentials with respect to $\omega$.

The linearisation of the scalar curvature map at $0$ is then given by
$$
dS_0 (f) = - L_{\omega} (f)  + \frac{1}{2}  \langle \nabla_{\omega} S(\omega), \nabla_{\omega} f \rangle_{\omega}.
$$
At a non-zero value of $\phi$, the linearisation is given by the same formula, but with respect to $\omega_{\phi}$ instead of $\omega$.

\subsection{Kuranishi theory}\label{sec:kuranishi}

The construction relies on Kuranishi theory (\cite{kuranishi65}). We follow the exposition of Sz\'ekelyhidi (\cite{szekelyhidi10}). Let $(M, \omega)$ be a symplectic manifold. We will denote by
$$
\mathcal{J} (M, \omega) = \{ J \in \End (TM) : J^2 = - \Id, J^* \omega = \omega \}
$$
the space of $\omega$-compatible almost complex structures on $M$. Here $J^* \omega (v,w) = \omega (Jv,Jw).$ The tangent space to $\mathcal{J} (M, \omega)$ at $J$ is given by
$$
T_J \mathcal{J} (M, \omega) = \{ A \in \End (TM ) : A \circ J = - J \circ A, \omega( A (\cdot), \cdot) = \omega( \cdot, A (\cdot)) \}.
$$
The space $\mathcal{J} (M, \omega)$ is in itself a complex manifold, with complex structure at $J$ given by $A \mapsto J \circ A$.

Kuranishi used the above framework to construct a \emph{versal deformation space} of a cscK $(\scX_0, \mathcal{L}_0)$. This is a complex space $V$ together with a holomorphic map
$$
f : V \to \mathcal{J} (M, \omega)
$$
coming from a universal family
$$
\mathcal{Y} \to V.
$$ 
If $T_0$ is a maximal compact subgroup $T_0$ of $\Aut_0 (\scX_0, \mathcal{L}_0)$, then this can be taken to be equivariant under the $T_0$ action. We assume that $(\scX_0, \mathcal{L}_0)$ is the central fibre of a test configuration of some semistable $(X,L)$ and we can then also assume that $T_0$ contains a maximal compact subgroup of the reduced automorphism group of $(X,L)$.

As a smooth manifold, $\mathcal{Y}$ is simply $V \times M$, and the action is the product action. If we have some initial deformation of $(X,L)$ to $(\scX_0, \mathcal{L}_0)$, this then arises, near the central fibre, from a map to $V$, and so we obtain a family 
$$
(\scX, \mathcal{L}) \to \Delta
$$
over a disk $\Delta \subset \mathbb{C}$, containing $0$, such that the central fibre is $(\scX_0, \mathcal{L}_0)$, and the non-zero fibres are isomorphic to $(X,L)$. The fact that such a family exists follows by \cite{szekelyhidi10} and the fact that the central fibre remains $(\scX_0, \mathcal{L}_0)$ follows by uniqueness of cscK degenerations (\cite{chensun14}). The family $(\scX, \mathcal{L})$ has a local $\mathbb{C}^*$ action, and we can equivariantly glue the restriction of $\scX$ to $\Delta \setminus \{ 0 \}$ to $X \times \mathbb{C}$, with the trivial action on the first component. We thus obtain a test configuration 
$$
\pi : (\scX, \mathcal{L}) \to \mathbb{P}^1
$$
over $\mathbb{P}^1$. Over the disk $\Delta$, we have a relatively K\"ahler metric on $\scX$, given simply by the symplectic form $\omega$, which is cscK on the central fibre. Note that the stabiliser of the $T_0$ action on $V$ at the point corresponding to $(X,L)$ is a torus $T$ inside the automorphism group of $(X,L)$, and if we assume $T_0 / T \cong \mathbb{C}^*$, which is the $\mathbb{C}^*$ action on the total space of the test configuration, then we have a $T_0$ action on the total space $\scX$ of the test configuration.

\section{Extending metrics}
\label{sec:metricsextension}

Here we will explain how to obtain a family of global $S^1$-invariant metrics on $\scX$ from a local deformation. We will also study their fibrewise properties, and in particular show that the fibrewise average scalar curvature goes to $0$ in the family.

The test configurations we consider arise in the manner described in Section \ref{sec:kuranishi}. We have a trivialisation of $\scX$ over $\mathbb{C}^*$ and diffeomorphisms $f_t : M \to M$ such that $f_t^* J_t = J$ and K\"ahler metrics $ \alpha_t = f^*_t \omega = \omega_0 + \ddb \phi_t$, for some background K\"ahler form $\omega_0$ on $X$. Let $\chi  (s) : [0, \infty) \to [0,1]$ be a cut-off function which is $1$ when $s \leq \frac{1}{2}$ and $0$ when $s \geq \frac{3}{4}.$ Define a new $2$-form $\omega_{\varepsilon}$ by
$$
\omega_{\varepsilon} = \omega_0 + \ddb \left( \chi \left( \frac{|t|}{\varepsilon} \right) \phi_t + (1-\chi \left( \frac{|t|}{\varepsilon} \right) ) \phi_{\varepsilon}  \right).
$$
The restriction of $\omega_{\varepsilon}$ is smooth relative K\"ahler form. The K\"ahler condition on the fibres $\scX_t$ of the test configuration is immediate: on a given fibre, we are taking a convex combination of K\"ahler forms ($\chi$ is only dependent on the base variable, and so its derivatives won't enter into showing positivity in the vertical direction). 

Note that this form is equal to $\alpha_t$ when $|t| \leq \frac{\varepsilon}{2}$ and equal to $\alpha_{\varepsilon}$ when $|t| \geq \frac{3 \varepsilon}{4}$. With respect to the original structure, we then have a path $I_t$ of complex structures such that $I_t = J_t$ when $|t| \leq \frac{\varepsilon}{2}$ and equal to $J_{\varepsilon}$ when $|t| \geq \frac{3 \varepsilon}{4}$. Thus we can fill in to a family over $\mathbb{P}^1$ by putting $I_t = I_{\varepsilon}$ when $|t|> \varepsilon$ too. Put differently, the relatively K\"ahler form $\omega_{\varepsilon}$ defined originally on the test configuration over a disk around the origin in $\mathbb{C}$, can be extended to a relatively K\"ahler form, still denoted $\omega_{\varepsilon}$, on the total space of the compactified test configuration over $\mathbb{P}^1$.

A key estimate is the following.
\begin{proposition}
\label{prop:potentialbound}
There exists positive constants $C>0$ such that for all $\varepsilon >0$ sufficiently small,
$$
\| \phi_t - \phi_{\varepsilon} \|_{C^{4,\alpha} (X)} \leq C \varepsilon
$$
for $t \in [\frac{\varepsilon}{2}, \varepsilon ].$ 
\end{proposition}

The key to establishing this result is to relate $\phi_t - \phi_{\varepsilon}$ to $J_t - J_{\varepsilon}$, where we already know such bounds. We may assume that $\omega$ the Kuranishi map $V \to \mathcal{J} (M, \omega)$ is defined on the ball of radius $2$. In particular, for the fibre at $1$, the symplectic forms $\omega$ and $\omega_1$ are both K\"ahler, in the same class, on $X= (M, J_1)$. The upshot is that in the description of $\omega_t$ we can take the reference metric to actually be the K\"ahler metric whose K\"ahler form is $\omega$ (and whose complex structure is $J_1$). Of course, the required bound is independent of background metric chosen, as we are looking at the \emph{difference} of $\phi_t$ and $\phi_{\varepsilon}$. However, it is convenient to pick $\omega$ as the reference K\"ahler form. 
\begin{proof} 
We know that $J_t = f_t^* J$ for a diffeomorphism $f_t : M \to M$. This diffeomorphism is produced from Moser's trick (see e.g. \cite[Section 7.2]{cannasdasilva01}), and is the time $1$ flow of the vector field $\nu_t$ satisfying 
$$
J d \phi_t + \iota_{\nu_t} \omega_t = 0,
$$ 
the equation coming from the fact that $\omega$ and $\omega_t$ are related by 
$$
\omega_t = \omega + \ddb \phi_t = \omega + d( J d \phi_t).
$$ 
Thus the $C^0$ norm of $f_t \circ f_{\varepsilon}^{-1}$, which can be bounded in terms of that of the vector fields, can be bounded in terms of the $C^1$ and $C^2$ semi-norms of the $\phi_t - \phi_{\varepsilon}$ (the $C^0$-norm does not enter into the expression, we are always taking at least one derivative). The upshot is that we can mutually bound, independently of $t$, the quantities
$$
\| J_t - J_{\varepsilon} \|_{C^0}
$$
and
$$
| \phi_t - \phi_{\varepsilon}|_{C^1} + | \phi_t - \phi_{\varepsilon} |_{C^2},
$$
where we use the notation $|\cdot|_{C^i}$ for the semi-norm, so that $\| \cdot \|_{C^i} = \sum_{j=0}^i |\cdot |_{C^j}.$ 

If we then work under the normalization that the $\phi_t$ are of average $0$ with respect to $\omega$, so that the $\phi_t$ lie in the image of the Laplacian on the central fibre, then we can bound the $C^0$-norm by the $C^2$ seminorm. So, under this normalisation for the $\phi_t$, we can mutually bound
$$
\| J_t - J_{\varepsilon} \|_{C^0}
$$
and
$$
\| \phi_t -  \phi_{\varepsilon}\|_{C^2}.
$$
Similarly, we can mutually bound higher derivatives. In particular, we can mutually bound, independently of $t$, 
$$
\| J_t - J_{\varepsilon} \|_{C^{2,\alpha}}
$$
and
$$
\| \phi_t - \phi_{\varepsilon} \|_{C^{4,\alpha}}.
$$
This gives the required bound for $\phi_t - \phi_{\varepsilon}$, since both $J_t$ and $J_{\varepsilon}$ approach $J_0$ as $\varepsilon \to 0$.
\end{proof}

The next goal will be to understand the expansion of the scalar curvature on $\scX$, when choosing a certain relationship between the rate of the parameter $\varepsilon$ of the approximate solution and the parameter $k$ for the polarisation. To understand the dependence on the rate $\varepsilon$, we use techniques reminiscent of the blowup situation, since our method of gluing metrics over an annular region is the construction being utilised there. Specifically, we are following the line of proof given in \cite[Proposition 20, Lemma 21, Lemma 24]{szekelyhidi12}.

In the vertical directions, we have the following estimate for the initial metric $\omega_{\varepsilon}$.
\begin{lemma}[{\cite{kuranishi65, szekelyhidi10}}]
Let $\hat S$ be the average scalar curvature of the fibres of $\scX \to \mathbb{P}^1$. Then
$$
S(\alpha_{\varepsilon}) - \hat S = O(\varepsilon).
$$
\end{lemma}
In fact, from \cite[Proof of Proposition 8]{szekelyhidi10}, this can be improved so that $S(\alpha_{\varepsilon}) - \hat S = O(\varepsilon^2)$, and, moreover, the $O(\varepsilon^2)$-term is the restriction to the fibre of a holomorphy potential of the $S^1$-action on the test configuration $\scX \to \mathbb{P}^1$. However, we will not need this for our construction.

We next want to show that the perturbed K\"ahler forms $\omega_{\varepsilon}$ satisfy a similar bound. From the mean value theorem, the scalar curvature of $\omega_{ \varepsilon}$ can be controlled using the linearisation. More precisely, we have that 
$$
\omega_{ \varepsilon}|_{\scX_t} = \alpha_{\varepsilon} + \ddb \left( \chi \left( \frac{|t|}{\varepsilon} \right) (\phi_t - \phi_{\varepsilon})  \right) .
$$
Therefore, by the mean value theorem,
$$
S(\omega_{\varepsilon}) = S(\alpha_{\varepsilon}) + L_{r}  \left( \chi \left( \frac{|t|}{\varepsilon} \right) (\phi_t - \phi_{\varepsilon})  \right) ,
$$
where $L_r$ is the linearisation of the scalar curvature operator at the metric
$$
\alpha_{\varepsilon} + r \cdot \ddb \left( \chi \left( \frac{|t|}{\varepsilon} \right) (\phi_t - \phi_{\varepsilon})  \right),
$$ 
for some $r \in [0,1]$. 

Our goal is to show that $S(\omega_{\varepsilon})$  is sufficiently close to $S(\alpha_{\varepsilon})$. The above shows that it will be key to understand the linearised operator of a perturbation of $\omega_{\varepsilon}$, which we now do. We emphasize that the below is a \emph{fibrewise} statement, as the $\alpha_{\varepsilon}$ are not even K\"ahler, only relatively K\"ahler.
\begin{proposition}
\label{prop:linearisationchange}
There exists $c, C>0$ such that if $\| \phi \|_{C^{4,\alpha}} \leq c$, then 
$$
\| L_{\alpha_{\varepsilon} + \ddb \phi} (f) - L_{\alpha_{\varepsilon}}  (f) \|_{C^{0,\alpha}} \leq C \| \phi \|_{C^{4,\alpha}} \| f \|_{C^{4,\alpha}}
$$
on each fibre of $\scX \to \mathbb{P}^1$.
\end{proposition}
\begin{proof}
This is a direct consequence of the fact that, with respect to some/any background metric $\omega_0$, we have the uniform lower bound
$$
\alpha_{\varepsilon} \geq C \omega_0.
$$
This follows because each member of the corresponding family of metrics $\alpha_{\varepsilon} (J(\cdot), \cdot)$ are isometric to $\omega (J_{\varepsilon} (\cdot), \cdot),$ and this is a smooth family of metrics, for a compact set of parameter values (the \emph{metric} also makes sense when $\varepsilon = 0$). The result then follows by expanding the metric quantities involved in the expression of $L_{\alpha_{\varepsilon}}$, using standard techniques e.g. as in \cite[Lemma 8.13]{szekelyhidi14book}.
\end{proof}

The above proposition reduces the bound on $S(\omega_{\varepsilon})$ to understanding the mapping properties of $\alpha_{\varepsilon}$. For now all we need is that $L_{\alpha_{\varepsilon}}$ is an operator that is uniformly bounded in $\varepsilon$ on the fibres. However, we prove a more refined result that will be need later. Let $\overline{\mathfrak{h}}$ be the space of holomorphy potentials on the cscK central fibre $\scX_0$, with respect to $\omega$. Then $\overline{\mathfrak{h}} = \langle h, 1 \rangle$, for some average zero holomorphy potential $h$ on $\scX_0$. We then have the following.
\begin{proposition}
\label{prop:fibrewiselinear}
There is an expansion 
$$
L_{\alpha_{\varepsilon}} = L_{\omega} +  O(\varepsilon),
$$
where $L_{\omega}$ is the linearisation of the scalar curvature of $\omega$ at $J_0$. In particular, $
L_{\alpha_{\varepsilon}}$ is surjective modulo $\overline{\mathfrak{h}}$ to leading order.
\end{proposition}
\begin{proof}
The scalar curvature of $\alpha_{\varepsilon}$ is constant to order $\varepsilon$, so it suffices to establish the above for the Lichnerowicz operator. 

We have that the Lichnerowicz operator is $\mathcal{D}_{\varepsilon}^*\mathcal{D}_{\varepsilon}$, where 
$$
\mathcal{D}_{\varepsilon} = \dbar_{J_{\varepsilon}} (\nabla^{1,0}_{\alpha_{\varepsilon}} ).
$$
We therefore have that
$$
\mathcal{D}_{\varepsilon} - \mathcal{D}_{0} = O(\varepsilon).
$$
This uses that the metrics $g_{\varepsilon}$ produced from the $\alpha_{\varepsilon}$ satisfy
$$
g_{\varepsilon} - g_0 = O(\varepsilon)
$$
and hence the same holds for the inverses, and the gradient. Again using this property, it follows that we have the same property for the adjoint, and hence for the full Lichnerowicz operator.
\end{proof}

With this in place, we finally have the estimate on the fibrewise scalar curvature that we need.
\begin{proposition}
Let $\hat S$ be the average scalar curvature of the fibres of $\scX \to \mathbb{P}^1$. Then
$$
S(\omega_{\varepsilon}|_{\scX_t}) - \hat S = O(\varepsilon).
$$
for any $t \in \mathbb{P}^1$.
\end{proposition}
\begin{proof}
From Proposition \ref{prop:potentialbound}, we know that $\phi_{\varepsilon} - \phi_t$ is $O(\varepsilon)$. Since we are interested in the scalar curvature on the fibre only, the term $\chi (\frac{|t|}{\varepsilon})$ is a constant. We can then apply Proposition \ref{prop:linearisationchange} to estimate $S(\omega_{\varepsilon}|_{X_t}).$ As remarked above, the scalar curvature is given by
$$
S(\omega_{\varepsilon}) = S(\alpha_{\varepsilon}) + L_{r}  \left( \chi \left( \frac{|t|}{\varepsilon} \right) (\phi_t - \phi_{\varepsilon})  \right) ,
$$
where $L_r$ is the linearisation of the scalar curvature operator at the metric
$$
\alpha_{\varepsilon} + r \cdot \ddb \left( \chi \left( \frac{|t|}{\varepsilon} \right) (\phi_t - \phi_{\varepsilon})  \right),
$$ 
for some $r \in [0,1]$. From Proposition \ref{prop:linearisationchange}, we therefore get that
\begin{align*}
&\| S(\omega_{\varepsilon}|_{X_t}) - \hat S \| \\
\leq & \|S(\alpha_{\varepsilon}) - \hat S \| + \| L_{r}  \left( \chi \left( \frac{|t|}{\varepsilon} \right) (\phi_t - \phi_{\varepsilon})  \right) \| \\
\leq & \|S(\alpha_{\varepsilon}) - \hat S \| + \| L_{0}  \left( \chi \left( \frac{|t|}{\varepsilon} \right) (\phi_t - \phi_{\varepsilon})  \right) \| +   \| (L_{r} -L_0) \left( \chi \left( \frac{|t|}{\varepsilon} \right) (\phi_t - \phi_{\varepsilon})  \right) \|  \\
\leq &  \|S(\alpha_{\varepsilon}) - \hat S \| + \| L_{0}  \left( \chi \left( \frac{|t|}{\varepsilon} \right) (\phi_t - \phi_{\varepsilon})  \right) \|  + C \| \phi_t - \phi_{\varepsilon} \|^2 .
\end{align*}
Since $L_0 = L_{\alpha_{\varepsilon}}$ is a bounded operator independently of $\varepsilon$, the above is an $O(\varepsilon)$ term, as required.
\end{proof}

\section{Constructing the extremal metrics}
\label{sec:construction}

In this section, we prove the main result of the paper, Theorem \ref{thm:mainepsilon}. We first relate the parameter $\varepsilon$ for the relative K\"ahler metrics on the test configuration $\scX$ to the parameter $k$ describing the polarisation. We then produce K\"ahler metrics on $\scX$, that moreover are approximately extremal. We continue to improve these approximate solutions, before finally perturbing to a genuine solution, when the approximate solutions are sufficiently good.

\subsection{K\"ahler metrics on the test configuration}
We now relate the parameters $k$ and $\varepsilon$. We define a closed $2$-form 
$$
\Omega_{\varepsilon} = \omega_{\varepsilon} + \lambda \varepsilon^{- \delta} \pi^* \omega_{FS}
$$
on $\scX$, where $\delta, \lambda > 0$ are constant parameters and $\omega_{FS}$ is the Fubini--Study metric on $\mathbb{P}^1$. So the relationship between the two parameters is $\varepsilon =(\lambda/k)^{\frac{1}{\delta}}$. A priori, this is just a closed $2$-form, but we will prove that when $\varepsilon$ is sufficiently small, it is K\"ahler, provided $\delta$ is sufficiently large so that it compensates for the horizontal contribution of $\omega_{\varepsilon}$.

Now,
\begin{align*}
\omega_{\varepsilon} =& \chi \left( \frac{|t|}{\varepsilon} \right) \alpha_t + (1-\chi \left( \frac{|t|}{\varepsilon} \right) )  \alpha_{\varepsilon} \\
& + i \partial \chi \left( \frac{|t|}{\varepsilon} \right) \wedge \bar \partial ( \phi_t - \phi_{\varepsilon} )  \\
&+ i \partial ( \phi_{\varepsilon} - \phi_t ) \wedge \bar \partial \chi \left( \frac{|t|}{\varepsilon} \right) \\
& + (\phi_t - \phi_{\varepsilon}) \ddb \chi \left( \frac{|t|}{\varepsilon} \right).
\end{align*}
The first line is a convex combination of relative K\"ahler forms, hence it is relatively K\"ahler. Since $J_{\varepsilon} - J_0 = O(\varepsilon)$, we therefore have that the horizontal contribution is $O(\varepsilon)$. We want to understand the contribution from the remaining terms, first to able to construct K\"ahler metrics, and then to understand the scalar curvature.

We first consider the term
$$
\partial \chi \left( \frac{|t|}{\varepsilon} \right) = \chi' \left( \frac{|t|}{\varepsilon} \right) \cdot \partial (|t|)/\varepsilon.
$$
Note that this is zero if $|t| < \frac{\varepsilon}{2},$ since $\chi'$ vanishes then. Note also that $\chi$ is uniformly bounded in $C^1$, as it is a fixed function. Finally, 
$$
\partial(|t|^2) = \bar{t} \partial (t),
$$
and so
$$
\partial (|t|) = \frac{1}{2|t|}  \partial( |t|^2) = \frac{\bar{t}}{2|t|} \partial (t),
$$
which is $O(1)$ for $t \in (\frac{\varepsilon}{2},\varepsilon)$. The term $\partial \chi \left( \frac{|t|}{\varepsilon} \right)$ is therefore an $O(\varepsilon^{-1})$-term. Similarly for $\bar{\partial} \chi (\frac{|t|}{\varepsilon}).$

Finally, we consider $ \ddb \chi \left( \frac{|t|}{\varepsilon} \right)$. Continuing from the calculations above, we have
\begin{align*}
 \ddb \chi \left( \frac{|t|}{\varepsilon} \right) =& - i\bar{\partial} \left( \chi' \left( \frac{|t|}{\varepsilon} \right) \cdot  \frac{\bar t}{2 \varepsilon |t|} \partial (t) \right) \\
=&- i\chi'' \left( \frac{|t|}{\varepsilon} \right) \cdot \frac{\bar{\partial} (|t|)}{\varepsilon} \wedge   \frac{\bar t}{2 \varepsilon |t|} \partial (t) - i\chi' \left( \frac{|t|}{\varepsilon} \right) \bar{\partial} \left(   \frac{\bar t}{2 \varepsilon |t|} \partial (t) \right)  \\
=& \frac{1}{4 \varepsilon^2 } \chi'' \left( \frac{|t|}{\varepsilon} \right) i \partial (t) \wedge  \bar{\partial} (\bar{t}) +  \frac{1}{4 \varepsilon |t|} \chi' \left( \frac{|t|}{\varepsilon} \right) i \partial (t) \wedge  \bar{\partial} (\bar{t})  
\\
=& \frac{1}{8} \left( \frac{1}{\varepsilon^2} \chi'' \left( \frac{|t|}{\varepsilon} \right) + \frac{1}{\varepsilon |t|} \chi' \left( \frac{|t|}{\varepsilon} \right) \right) i d t \wedge d (\bar{t} ),
\end{align*}
which is an $O(\varepsilon^{-2})$-term.

The above calculations allow us to prove the following.
\begin{lemma}
For any $\delta \geq 1$, there exists $\varepsilon_0, \lambda>0$ such that for all $\varepsilon \in (0,\varepsilon_0)$, $\Omega_{\varepsilon}$ is K\"ahler on $\scX$. When $\delta >1$, the same conclusion holds for any $\lambda$, with $\varepsilon_0$ depending on $\lambda$.
\end{lemma}
\begin{proof}
The $\omega_{\varepsilon}$ are relatively K\"ahler, so all that needs to be checked is the horizontal component with respect to the fibration structure 
$$
\pi : \scX \to \mathbb{P}^1.
$$
It follows from Proposition \ref{prop:potentialbound} that the horizontal component of $\omega_{\varepsilon}$ is $O(\varepsilon^{-1})$. More precisely, it is bounded below by a constant multiple of $- \varepsilon^{-1} \pi^* \omega_{FS}$ when $\varepsilon$ is sufficiently small. Thus if $\delta = 1$, we can pick $\lambda >0$ sufficiently large to make $\Omega_{\varepsilon}$ K\"ahler. If $\delta > 1$, we can do so for any $\lambda.$
\end{proof}

\subsection{Expansion of the scalar curvature}

We now want to understand the scalar curvature of the above metric.
\begin{proposition}
\label{prop:firstapproxerror}
Let $\hat S$ denote the average scalar curvature of the fibres of $\scX \to \mathbb{P}^1$. Then 
$$
S(\Omega_{\varepsilon}) - \hat S = O(\varepsilon^{\tau})
$$
for some $\tau >0$, which can be taken to be $1$ if $\delta \geq 2$. Moreover, the horizontal component of $S(\omega_{\varepsilon})$ satisfies 
$$
(S(\Omega_{\varepsilon}) - \hat S )_{\scH} = \frac{\varepsilon^{\delta}}{\lambda} S(\omega_{FS}) + O(\varepsilon^{\delta+ \tau'}),
$$
for some $\tau'>0$, which again can be taken to be $1$ if $\delta \geq 2$.
\end{proposition}
\begin{proof}
We follow the strategy of \cite[Lemma 3.3]{fine04} to compute the expansion of the scalar curvature. We begin with the Ricci curvature. The Ricci curvature of $\Omega_{\varepsilon}$ is the sum of the curvatures of the induced metrics on $\Lambda^{n} \scV$ and $\scH$, where $\scV= \ker \pi_*$ is the vertical subbundle of the tangent bundle and $\scH \cong \pi^* T \mathbb{P}^1$ is its orthogonal complement with respect to $\Omega_{\varepsilon}$, due to the short exact sequence
$$
0 \to \scV \to T \scX \to \pi^* T \mathbb{P}^1 \to 0 .
$$

We will denote by $\rho_{\varepsilon}$ the curvature on $\Lambda^n \scV$. Note that the vertical component $\rho_{\varepsilon, \scV}$ of this is nothing but the Ricci curvature of the restriction of $ \omega_{\varepsilon} $ to the fibres, so that $\Lambda_{\omega_{\varepsilon}} ( \rho_{\varepsilon}) = S(\omega_{\varepsilon})$ is the fibrewise scalar curvature. We will deal with its horizontal component $\rho_{\varepsilon, \scH}$ at the end of the proof.

For the curvature $F_{\scH} $ on $\scH$, we also have another metric, namely $\pi^* \omega_{FS}$. The curvature of this metric is simply the pullback of the curvature on $B$, and the two curvatures differ by an exact term. We therefore have that the curvatures compare as
\begin{align*}
iF_{\scH} - \Ric (\omega_{FS}) =& \ddb \log \left( \frac{\omega_{\varepsilon} + \lambda \varepsilon^{-\delta} \omega_{FS}}{\omega_{FS}}\right) \\
=& \ddb \log \left( 1 + \frac{\varepsilon^{\delta}}{\lambda} \Lambda_{\omega_{FS}} (\omega_{\varepsilon}) \right) \\
=&  \frac{\varepsilon^{\delta}}{\lambda} \ddb \Lambda_{\omega_{FS}} (\omega_{\varepsilon})  + O(\varepsilon^{2(\delta-1)}).
\end{align*}
Note that $ \frac{\varepsilon^{\delta}}{\lambda} \ddb \Lambda_{\omega_{FS}} (\omega_{\varepsilon})  $ is an $O(\varepsilon^{\delta -1 })$ term, since $\omega_{\varepsilon}$ is $O(\varepsilon^{-1})$. This is why the error term is $ O(\varepsilon^{2(\delta-1)})$ above.

The upshot is that we have the expansion
$$
\Ric (\Omega_{\varepsilon}) = \rho_{\varepsilon} + \Ric(\omega_{FS}) +  \frac{\varepsilon^{\delta}}{\lambda} \ddb \Lambda_{\omega_{FS}} (\omega_{\varepsilon}) + O(\varepsilon^{2(\delta-1)}) .
$$

Next, we contract to obtain the scalar curvature. We have
$$
\Lambda_{\Omega_{\varepsilon}} = \Lambda_{\scV} + \frac{\varepsilon^{\delta}}{\lambda} \Lambda_{\omega_{FS}} + O(\varepsilon^{\delta + 1}),
$$
since the restriction of the metric to $\scV$ is the restriction of $\omega_{\varepsilon}$, and the metric on $\scH$ is $\lambda \varepsilon^{-\delta} \omega_{FS}$ to leading order, the subleading order terms coming from the horizontal component of $\omega_{\varepsilon}$.  Therefore,
\begin{align*}
S(\Omega_{\varepsilon}) =& S(\omega_{\varepsilon}) + \varepsilon^{\delta} \Lambda_{\omega_{FS}} (\rho_{\varepsilon, \scH}) + \frac{\varepsilon^{\delta}}{\lambda}  S(\omega_{FS}) + \frac{\varepsilon^{\delta}}{\lambda} \Delta_{\scV} (\Lambda_{\omega_{FS}} (\omega_{\varepsilon}) ) + O(\varepsilon^{\textnormal{min} \{ 3\delta-2, \delta + 1 \} }).
\end{align*}
Note that $\frac{\varepsilon^{\delta}}{\lambda} \Delta_{\scV} (\Lambda_{\omega_{FS}} (\omega_{\varepsilon}) ) $ is a vertical $O(\varepsilon^{\delta-1})$ term. Since $\delta > 1$, this is therefore a decaying vertical term, and is $\varepsilon$ to an integer power if $\delta > 1$ is an integer. 

The term $\rho_{\varepsilon, \scH}$ is the horizontal component of the curvature of $\Lambda^n \scV$ induced from $\omega_{\varepsilon}$. We claim that $\Lambda_{\omega_{FS}} (\rho_{\varepsilon, \scH})$ decays with $\varepsilon$, which gives the required expansion of the scalar curvature. Now, $\rho_{\varepsilon}$ is determined by restricting $\omega_{\varepsilon}$ to a fibre $\scX_t$ and then taking the Ricci curvature. When $t > \varepsilon$, $\omega_{\varepsilon}$ restricted to $\scX_t$ is simply $\alpha_{\varepsilon}$, which is constant in the horizontal direction, and so the horizontal contribution is entirely contained in the ball of radius $\varepsilon$ in the base direction. In this region, $\omega_{\varepsilon}$ is a convex combination of $\alpha_{t}$ and $\alpha_{\varepsilon}$, both of which are $O(\varepsilon)$ perturbations of $\omega$. Thus, to leading order in $\varepsilon$, $\rho_{\varepsilon}$ is simply the Ricci curvature of the central fibre, and so does not vary in the horizontal direction. Therefore, the term $\Lambda_{\omega_{FS}} (\rho_{\varepsilon, \scH})$ decays with $\varepsilon$, and so the  required expansion holds.
\end{proof}

\begin{remark}
The above explains why we picked the Fubini--Study metric on $\mathbb{P}^1$. This means that the leading order term in the horizontal direction is a constant.
\end{remark}

\subsection{Improving the approximate solution}

Proposition \ref{prop:firstapproxerror} shows that $\Omega_{\varepsilon}$ is approximately extremal. We now want to show that we can perturb to a genuine solution of the extremal equation, for sufficiently small $\varepsilon$. The next step is to improve the approximate solution $\Omega_{\varepsilon}$.

We first describe the holomorphy potential for the $S^1$-action on $\scX$ with respect to our metrics. Recall that $\chi$ was the cut-off function used in the definition of the relative K\"ahler $\omega_{\varepsilon}$.
\begin{lemma}
\label{lem:holpots} Let $h_{\varepsilon}$ be the average $0$ holomorphy potential of the $S^1$-action on $\scX$ with respect to $\Omega_{\varepsilon}$. Then, after potentially changing the potentials $\phi_t$ by a function pulled back from $\mathbb{C}$, we have that 
$$
h_{\varepsilon} = \chi (t/\varepsilon) \cdot  h_0  + \lambda \varepsilon^{-\delta} h_{FS},
$$
where 
\begin{itemize} 
\item $h_0$ is the average $0$ potential with respect to the central fibre, thought of as a vertical function on $\scX$;
\item $h_{FS}$ is the  average $0$  potential of the corresponding action on $\mathbb{P}^1$ with respect to $\omega_{FS}$, pulled back to $\scX$.
\end{itemize}
\end{lemma}
\begin{proof}
Note that $\Omega_{\varepsilon}$ is an $S^1$-invariant K\"ahler form on $\scX$, and as such we know the existence of $h_{\varepsilon}$. Note that because of the linearity of the holomorphy potential of a fixed vector field with respect to the symplectic form, we therefore have the existence of a potential with respect to the \emph{relative} K\"ahler form $\omega_{\varepsilon}$. Note that the notion of a hamiltonian makes sense even if a $2$-form is not symplectic -- it is just not guaranteed to exist. Now, if $\eta_{\varepsilon}$ denotes this potential with respect to $\omega_{\varepsilon}$, it suffices to show that $\eta_{\varepsilon} = \chi h_0 $, again by the linearity of the potential ($\lambda \varepsilon^{-\delta} h_{FS}$ is the potential with respect to $ \lambda \varepsilon^{- \delta} \pi^* \omega_{FS}$).

We now prove that $\eta_{\varepsilon} = \chi (t/\varepsilon) \cdot h_0$. Since $h_0$ is a holomorphy potential for the generator $\nu$ of the $S^1$-action on the central fibre, we have that
$$
d_{\scX_0} (h_0 ) = \iota_{\nu} \omega.
$$
Recall the description of $\omega_{\varepsilon}$. We have a $\mathbb{C}^*$ equivariant biholomorphism $F :  (M, J) \times \mathbb{C}^* \to \scX_{|\pi^{-1} \mathbb{C}^*} $, giving diffeomorphisms $f_t = F_ {|\scX_t}: M \to M$ such that $f_t^* J_t = J$ and K\"ahler metrics $ \alpha_t = f^*_t \omega = \omega + \ddb \phi_t$, for some background K\"ahler form $\omega$ on $X$. The cut-off function $\chi  (s) : [0, \infty) \to [0,1]$ is $1$ when $s \leq \frac{1}{2}$ and $0$ when $s \geq \frac{3}{4}$, and $\omega_{\varepsilon}$ is defined by
$$
\omega_{\varepsilon} = \omega + \ddb \left( \chi \left( \frac{|t|}{\varepsilon} \right) \phi_t + (1-\chi \left( \frac{|t|}{\varepsilon} \right) ) \phi_{\varepsilon}  \right).
$$
Note that this is exactly equal to $\alpha_{\varepsilon}$ for all fibres $t$ with $t \geq \frac{3}{4} \varepsilon.$ It follows that $(\iota_{\nu} \omega_{\varepsilon})_{ \pi^{-1} ( \mathbb{P}^1 \setminus B_{\frac{3}{4} \varepsilon})} = 0$ and so we have $(h_{\varepsilon})_{|\pi^{-1} ( \mathbb{P}^1 \setminus B_{\frac{3}{4} \varepsilon})} = 0$. In other words, $\eta_{\varepsilon}$ is supported on the preimage of the ball of radius $\frac{3}{4} \varepsilon$.

Using the expansion
\begin{align*}
\omega_{\varepsilon} =& \chi \left( \frac{|t|}{\varepsilon} \right) \alpha_t + (1-\chi \left( \frac{|t|}{\varepsilon} \right) )  \alpha_{\varepsilon} \\
& + i \partial \chi \left( \frac{|t|}{\varepsilon} \right) \wedge \bar \partial ( \phi_t - \phi_{\varepsilon} )  \\
&+ i \partial ( \phi_{\varepsilon} - \phi_t ) \wedge \bar \partial \chi \left( \frac{|t|}{\varepsilon} \right) \\
& + (\phi_t - \phi_{\varepsilon}) \ddb \chi \left( \frac{|t|}{\varepsilon} \right).
\end{align*}
we have that if we pull back to $\scX$, then 
\begin{align*}
(F^{-1})^* ( \omega_{\varepsilon}) =& \chi \left( \frac{|t|}{\varepsilon} \right) \omega + (1-\chi \left( \frac{|t|}{\varepsilon} \right) )  \tau \\
&+ i (F^{-1})^* \partial \chi \left( \frac{|t|}{\varepsilon} \right) \wedge \bar \partial ( \phi_t - \phi_{\varepsilon} )  \\
&+ i (F^{-1})^*  \partial ( \phi_{\varepsilon} - \phi_t ) \wedge \bar \partial \chi \left( \frac{|t|}{\varepsilon} \right) \\
& + (F^{-1})^*  (\phi_t - \phi_{\varepsilon}) \ddb \chi \left( \frac{|t|}{\varepsilon} \right).
\end{align*}
Here $\tau$ is the pullback of $\alpha_{\varepsilon}$, extended to a form on the whole of $(M,J) \times \mathbb{C}^*$. It is not moving in the horizontal directions, so $\iota_{\nu} \tau = 0$. Similarly, $\iota_{\nu} \ddb \chi \left( \frac{|t|}{\varepsilon} \right)  = 0$.

Next, we have that 
\begin{align*}
\iota_{\nu} (F^{-1})^* \ddb (\chi \phi_t) =& - i \partial \chi \wedge \iota_{\nu} (F^{-1})^* \bar \partial \phi_t + i \bar \partial \chi \wedge \iota_{\nu} (F^{-1})^*  \partial \phi_t \\
&+ \chi \iota_{\nu} (F^{-1})^* \ddb (\phi_t) \\
=& - i \partial \chi \wedge \iota_{\nu} (F^{-1})^* \bar \partial \phi_t + i \bar \partial \chi \wedge \iota_{\nu} (F^{-1})^*  \partial \phi_t \\
&+ \chi dh \\
=& - i \partial \chi \wedge \iota_{\nu} (F^{-1})^* \bar \partial \phi_t + i \bar \partial \chi \wedge \iota_{\nu} (F^{-1})^*  \partial \phi_t \\
&+ d (\chi h) - h d \chi \\
=&  - \partial \chi \wedge (i \iota_{\nu} (F^{-1})^* \bar \partial \phi_t  - h ) +  \bar \partial \chi \wedge ( i \iota_{\nu} (F^{-1})^*  \partial \phi_t + h) \\
&+ d (\chi h) - h d \chi.
\end{align*}

From 
$$
dh = \iota_{\nu} (F^{-1})^* \ddb (\phi_t)
$$
we know that 
$$
h -i \iota_{\nu} (F^{-1})^* \bar \partial \phi_t 
$$
and 
$$
h + i \iota_{\nu} (F^{-1})^*  \partial \phi_t 
$$
are closed.  Now, $\nu$ generates an $S^1$ action, and if we restrict to the set $|t|= r$, for $r>0$, then this is compact (real) manifold. In particular, the above two expressions have to be constant. In other words, there exists constants $c_{|t|}$ such that 
$$
h + i \iota_{\nu} (F^{-1})^*  \partial \phi_t  = c_{|t|},
$$
and similarly
$$
h -i \iota_{\nu} (F^{-1})^* \bar \partial \phi_t  = c_{|t|}
$$
since this is just the complex conjugate of the former.

We can then get rid of this constant by changing $\phi_t$ by a constant, fibrewise. This of course changes nothing for the metric on the fibre, but changes the global expression on the total space, by $\ddb$ of a function pulled back from $\mathbb{P}^1$, which moreover is $S^1$ invariant, as it depends on $|t|$ only. The result follows.
\end{proof}

 Proposition \ref{prop:firstapproxerror} implies that the error is all vertical until order $\varepsilon^{\delta}$. We therefore need to correct vertical errors until this order. To make sure we are not introducing horizontal errors interfering with the contribution from $\omega_{FS}$, we need to understand the mapping properties of the linearised operator in more detail.
\begin{proposition}
\label{prop:linearisation}
Suppose $\delta \geq 1$. The linearisation $L_{\varepsilon}$ of the scalar curvature operator at $\Omega_{\varepsilon}$ admits an expansion
$$
L_{\varepsilon} = P_{\varepsilon} + \varepsilon^{2\delta} R_{\varepsilon},
$$
where the image of $P_{\varepsilon}$ is vertical and $R_{\varepsilon}$ is a bounded operator independent of $\varepsilon$. Further, $P_{\varepsilon}$ and the horizontal component of $R_{\varepsilon}$ admit expansions
$$
P_{\varepsilon} = L_{\omega} + O(\varepsilon)
$$
for any $f \in C^{\infty} (\scX)$, where $L_{\omega} = - \mathcal{D}^*_{\omega} \mathcal{D}_{\omega}$ is the linearisation of the scalar curvature at the central fibre, and
$$
R_{\varepsilon} ( \pi^* f) = - \frac{1}{\lambda^2} \cdot \pi^* \left( \mathcal{D}^*_{\omega_{FS}} \mathcal{D}_{\omega_{FS}} (f) \right) + O(\varepsilon)
$$
for $f  \in C^{\infty} ( \mathbb{P}^1)$.

These expansions persist upon perturbations of $\Omega_{\varepsilon}$ by potentials $\phi_{\varepsilon}$ which are 
\begin{itemize}
\item $O(\varepsilon^{\tau})$ for $\tau>0$, if $\phi_{\varepsilon} \in C^{\infty}_0 (X)$;
\item $O(\varepsilon^{\tau - \delta})$ for $\tau > 0$, if $\phi_{\varepsilon} \in \pi^* C^{\infty} ( \mathbb{P}^1).$
\end{itemize}
\end{proposition}
\begin{proof}
We use the standard expression
\begin{align*}
L_{\varepsilon} (f) &= - \mathcal{D}^*_{\Omega_{\varepsilon}} \mathcal{D}_{\Omega_{\varepsilon}} (f) + \frac{1}{2} \langle \nabla S(\Omega_{\varepsilon}), \nabla f \rangle
\end{align*}
for the linearisation of the scalar curvature. 

We begin with the leading order term of the expansion. First note that, since, by Proposition \ref{prop:firstapproxerror}, $S(\Omega_{\varepsilon}) = \hat S + O(\varepsilon)$, the gradient term does not enter to leading order in $\varepsilon$. We are therefore left with expanding the Lichnerowicz operator. Now, the Lichnerowicz operator admits an expansion
\begin{align*}
L_{\Omega_{\varepsilon}} (f) = \Delta^2_{\Omega_{\varepsilon}} (f) + \langle \Ric (\Omega_{\varepsilon}) , \ddb (f) \rangle_{\Omega_{\varepsilon}} + \frac{1}{2}  \langle \nabla S(\Omega_{\varepsilon}), \nabla f \rangle_{\Omega_{\varepsilon}}.
\end{align*}
So, we need to understand the expansion of the Laplacian, the Ricci curvature and the scalar curvature.

We begin with the Laplacian. The Laplacian operator can be written 
$$
\Delta_{\Omega_{\varepsilon}} (f) \Omega_{\varepsilon}^{n+1} = (n+1) \ddb (f) \wedge \Omega_{\varepsilon}^n.
$$
since the base $\mathbb{P}^1$ is one-dimensional, we have expansions
\begin{align*}
\Omega_{\varepsilon}^{n+1} =&   \omega_{\varepsilon}^{n+1}  + (n+1) \lambda \varepsilon^{- \delta} \pi^* \omega_{FS} \wedge \omega_{\varepsilon}^n \\
\Omega_{\varepsilon}^{n} =&  \omega_{\varepsilon}^{n}  + n \lambda \varepsilon^{- \delta} \pi^* \omega_{FS} \wedge \omega_{\varepsilon}^{n-1} 
\end{align*}
Thus the leading order term of the Laplacian equation is
$$
 \Delta_{\Omega_{\varepsilon}} (f) \pi^* \omega_{FS} \wedge \omega_{\varepsilon}^{n} =  n \ddb (f) \wedge \pi^* \omega_{FS} \wedge \omega_{\varepsilon}^{n-1} + O(\varepsilon^{-1}).
$$
In the leading order term, we therefore only get a vertical contribution from $\ddb (f)$. This gives that 
\begin{align*}
\Delta_{\Omega_{\varepsilon}} (f) =&  \Lambda_{\omega_{\varepsilon}} ( \ddb_{\mathcal{V}} (f) ) + O(\varepsilon) \\
=& \Delta_{\scX_0} (f) + O(\varepsilon).
\end{align*}
Here we are also using that $J_{\varepsilon} - J_0 = O(\varepsilon)$ so that we get a similar bound for the difference of the $\ddb$-operators near the central fibre, and hence everywhere (since the family of K\"ahler manifolds is constant for a given $\varepsilon$ away from a small neighbourhood of the central fibre). 

For the Ricci term, we have from the proof of Proposition \ref{prop:firstapproxerror} that
$$
\Ric (\Omega_{\varepsilon}) = \rho_{\varepsilon} + \Ric(\omega_{FS}) + O(\varepsilon).
$$
Since the inner product $ \langle \Ric (\Omega_{\varepsilon}) , \ddb (f) \rangle_{\Omega_{\varepsilon}}$ equals the inner product with respect to the vertical metric from the central fibre to leading order, we therefore have 
$$
 \langle \Ric (\Omega_{\varepsilon}) , \ddb (f) \rangle_{\Omega_{\varepsilon}} =  \langle \Ric (\omega) , \ddb (f) \rangle_{\omega} + O(\varepsilon).
$$

Finally, the scalar curvature term vanishes to leading order, since $S(\Omega_{\varepsilon})$ is constant to leading order. This term also does not enter in the Lichnerowicz operator on the central fibre, since the central fibre has constant scalar curvature.

Together with the above, this shows that the leading order term of the expansion above is the negative of the Lichnerowicz operator of the metric on the central fibre, which gives the expansion of $P_{\varepsilon}$ (since $\delta \geq 1$, and so all the errors are $O(\varepsilon)$).

Next, we have a refined expansion when the function is of the form $\pi^* (f)$, i.e. the function is pulled back from $\mathbb{P}^1$. In this case, the expansion of the Laplacian becomes
$$
\lambda \varepsilon^{- \delta} \Delta_{\Omega_{\varepsilon}} (f) \pi^* \omega_{FS} \wedge \omega_{\varepsilon}  =  n \ddb(f) \wedge \omega_{\varepsilon}^{n} + O(\varepsilon)
$$
where we are omitting pullback from the notation. The above equation follows because $n \ddb (f) \wedge \pi^* \omega_{FS} \wedge \omega_{\varepsilon}^{n-1}  = 0$, as f is pulled back from $\mathbb{P}^1$. It follows that 
$$
\Delta_{\Omega_{\varepsilon}} (\pi^* (f) ) = \frac{\varepsilon^{\delta}}{\lambda} \cdot \pi^* (\Delta_{\omega_{FS}} (f) ) + O(\varepsilon^{\delta+1}).
$$
and so
$$
\Delta_{\Omega_{\varepsilon}}^2 (\pi^* (f) ) = \frac{\varepsilon^{2\delta}}{\lambda^2} \cdot \pi^* (\Delta^2_{\omega_{FS}} (f) ) + O(\varepsilon^{2 \delta+1}).
$$

For the Ricci curvature, we now have
\begin{align*}
 \langle \Ric (\Omega_{\varepsilon}) , \ddb (f) \rangle_{\Omega_{\varepsilon}} =&  \langle \Ric (\omega_{FS}) , \ddb (f) \rangle_{\lambda \varepsilon^{-\delta} \omega_{FS} + (\omega_{\varepsilon})_{\scH} } + O(\varepsilon^{3 \delta}) \\
=& \frac{\varepsilon^{2\delta}}{\lambda^2} \langle \Ric (\omega_{FS}) , \ddb (f) \rangle_{\omega_{FS} } + O(\varepsilon^{3 \delta}) 
\end{align*}
Note that here we are using the induced inner product on $2$-forms, which scales like the scaling factor to the power $-2$, hence the equation in the second line and that the error is $O(\varepsilon^{3 \delta})$.

Again, the scalar curvature term does not enter either in the expansion or in the model on $\mathbb{P}^1$, as $\omega_{FS}$ (though, if we had pulled back another metric, we would have a matching of these terms anyway, so this does not use that we have the Fubini--Study metric in any essential way). Thus the claimed expansion on pulled back functions also holds.

Finally, we need to check that the expansions persist under certain perturbations of $\Omega_{\varepsilon}$. It is immediate that the $O(1)$-term of the expansion will be unchanged upon perturbing $\Omega_{\varepsilon}$ by $\varepsilon^{\tau} \ddb \phi$, for any $\phi$. We can also allow a negative power of $\varepsilon$ and keep the same expansion, when $\phi = \pi^* f$ is pulled back from $\mathbb{P}^1$. The vertical component of the metric is then unaffected. In the horizontal direction, we then need to be perturbing $\omega_{FS}$, which we do provided the exponent is strictly larger than $-\delta$. This completes the proof. 
\end{proof}

As used at the end of the proof above, we can also perturb $\omega_{FS}$ before pulling back. This perturbs $\Omega_{\varepsilon}$ by terms that at first glance look like they blow up.
\begin{lemma}
\label{lem:linhorizontal}
Let $f \in C^{\infty} (\mathbb{P}^1)$. Then for any $\tau > 0$, we have that 
$$
S\left( \Omega_{\varepsilon} + \varepsilon^{\tau- \delta} \ddb \pi^* (f) \right) = S(\Omega_{\varepsilon}) - \frac{\varepsilon^{\delta + \tau}}{\lambda^2} \pi^* \left( \mathcal{D}^*_{\omega_{FS}} \mathcal{D}_{\omega_{FS}} (f) \right)  + O(\varepsilon^{\delta + \tau + 1}).
$$
\end{lemma}
\begin{proof}
This is exactly the same as the computations for the horizontal terms in Proposition \ref{prop:linearisation}.
\end{proof}

\begin{remark}
Of course, the above lemma can be applied iteratively. If we define $\Omega_{\varepsilon}' =  \Omega_{\varepsilon} + \varepsilon^{\tau- \delta} \ddb \pi^* (f)$ and then perturb $\Omega_{\varepsilon}'$ to $\Omega_{\varepsilon}' + \varepsilon^{\tau'-\delta} \ddb ( \pi^* \phi)$ for some $\tau'>\tau$, then 
$$
S \left( \Omega'_{\varepsilon} + \varepsilon^{\tau'- \delta} \ddb \pi^* (\phi) \right) = S \left( \Omega_{\varepsilon}' \right) - \frac{\varepsilon^{\delta + \tau'}}{\lambda^2} \pi^* \left( \mathcal{D}^*_{\omega_{FS}} \mathcal{D}_{\omega_{FS}} (\phi) \right)  + O(\varepsilon^{\delta + \tau' + 1}).
$$
Similar statements hold when we iterate multiple times, and also when we further perturb $\Omega_{\varepsilon}'$ by functions not necessarily pulled back from $\mathbb{P}^1$, but which are decaying with $\varepsilon$.
\end{remark}

Having this understanding of the linearised operator, we can now improve the approximate solutions to be extremal to as high order in $\varepsilon$ as we would like.
\begin{proposition}
\label{prop:approxsoln}
For any $\kappa \geq 0$, there exists a K\"ahler form $\Omega_{\varepsilon,\kappa}$ on $\scX$ and a holomorphy potential $h_{\varepsilon, \kappa}$ with respect to  $\Omega_{\varepsilon,\kappa}$ such that
$$
S \left( \Omega_{\varepsilon,\kappa} \right) = h_{\varepsilon, \kappa} + O(\varepsilon^{\kappa+1}).
$$
\end{proposition}
\begin{proof}
For $\kappa=0$, we take $\Omega_{\varepsilon,0} = \Omega_{\varepsilon}$. Proposition \ref{prop:firstapproxerror} then shows that we have the required expansion, with $h_{\varepsilon,0} = \hat S$ a constant, which is a potential for the trivial holomorphic vector field on $\scX$. 

We proceed to construct better approximate solutions inductively. For simplicity we will assume that $\delta$ is an integer. This is not essential, and is simply for notational convenience. We can then do induction over $\kappa \in \mathbb{Z}_{\geq 0}$ and assume that all terms appear at integer values of $\varepsilon$. Ultimately we can choose $\delta =2$ to make the construction work, so we lose nothing in the proof of the main result in assuming this.

We will inductively show that we can find
\begin{itemize}
\item functions $f_0, \ldots, f_{\kappa} \in C^{\infty}_0 (\scX)$;
\item functions $\phi_0, \ldots, \phi_{\kappa} \in C^{\infty} (\mathbb{P}^1)$;
\item holomorphy potentials $h_0, \ldots, h_{\kappa}$ with respect to $\Omega_{\varepsilon}$;
\end{itemize}
such that, if we put 
$$
\Omega_{\varepsilon, \kappa} = \Omega_{\varepsilon} + \ddb\left( \sum \varepsilon^j f_j + \sum_j \varepsilon^{j-\delta} \pi^* \phi_j \right)
$$
then $\Omega_{\varepsilon, \kappa}$ satisfies
$$
S \left( \Omega_{\varepsilon, \kappa} \right) = h_{\varepsilon, \kappa} + O(\varepsilon^{\kappa +1})
$$
where
$$
h_{\varepsilon,\kappa} = \sum_{j=0}^{\kappa} \varepsilon^{j} h_j + \left\langle \nabla \left(  \sum_{j=0}^{\kappa} \varepsilon^j h_j \right), \nabla \left( \sum_{j=0}^{\kappa} \varepsilon^j f_j + \sum_{j=0}^{\kappa} \varepsilon^{j-\delta} \pi^* \phi_j \right) \right\rangle,
$$
which is a holomorphy potential with respect to $\Omega_{\varepsilon, \kappa}$. As remarked above, for $\kappa = 0$, we have the required statement with $f_0 = 0 = \phi_0$ and $h_0 = \hat S$, the average scalar curvature of the fibres.

The argument differs depending on whether or not $\kappa \geq \delta$. So we begin with $\kappa < \delta$. In this case, we will additionally show inductively that there is no horizontal term up to order $\varepsilon^{\delta}$ and that the horizontal term at order $\varepsilon^{\delta}$ is precisely given by the horizontal term in the expansion of Proposition  \ref{prop:firstapproxerror}.
 
By the induction assumption, 
$$
S \left( \Omega_{\varepsilon, \kappa} \right) - h_{\varepsilon, \kappa} = l_{\kappa +1} \varepsilon^{\kappa + 1} + O(\varepsilon^{\kappa + 2})
$$
for some function $l_{\kappa + 1} \in C^{\infty}_0 (\scX)$. By the mapping properties of $L_{\omega}$, there exists $f_{\kappa +1} \in C^{\infty}_{0} (\scX)$ and a holomorphy potential $h_{\kappa +1}$ such that 
$$
L_{\omega} (f_{\kappa +1} ) = - l_{\kappa +1} + h_{\kappa +1}.
$$
By Proposition \ref{prop:linearisation} applied to $ \Omega_{\varepsilon, \kappa}$, we therefore have that 
\begin{align*}
& S \left( \Omega_{\varepsilon, \kappa} + \ddb ( \varepsilon^{\kappa + 1} f_{\kappa+1} ) \right) \\
=& S( \Omega_{\varepsilon, \kappa} ) - \varepsilon^{\kappa+1} l_{\kappa +1}  +  \varepsilon^{\kappa +1 } h_{\kappa+1} + O(\varepsilon^{\kappa+2}) \\
=& \sum_{j=0}^{\kappa+1} \varepsilon^{j} h_j +  \left\langle \nabla \left(  \sum_{j=0}^{\kappa} \varepsilon^j h_j \right), \nabla \left( \sum_{j=0}^{\kappa} \varepsilon^j f_j + \sum_{j=0}^{\kappa} \varepsilon^{j-\delta} \pi^* \phi_j \right)  \right\rangle + O(\varepsilon^{\kappa+2}).
\end{align*}
Moreover, since the terms 
$$
 \left\langle \nabla \left(   \varepsilon^{\kappa +1} h_j \right), \nabla \left( \sum_{j=0}^{\kappa+1} \varepsilon^j f_j + \sum_{j=0}^{\kappa+1} \varepsilon^{\delta-j} \pi^* \phi_j \right) \right\rangle
$$
and
$$
 \left\langle \nabla \left(  \sum_{j=0}^{\kappa +1} \varepsilon^j h_j \right), \nabla \left(  \varepsilon^{\kappa+1} f_{\kappa+1} \right) \right\rangle
$$
are $O(\varepsilon^{\kappa+2})$ (since the $\kappa=0$ terms are all constant), we have that the above agrees with $h_{\varepsilon, \kappa+1}$ at order $\varepsilon^{\kappa+1}$, i.e.
\begin{align*}
 S \left( \Omega_{\varepsilon, \kappa} + \ddb ( \varepsilon^{\kappa + 1} f_{\kappa+1} ) \right) =&h_{\varepsilon, \kappa+1} + O(\varepsilon^{\kappa+2}),
\end{align*}
where
$$
h_{\varepsilon,\kappa+1} = \sum_{j=0}^{\kappa+1} \varepsilon^{j} h_j + \left\langle \nabla \left(  \sum_{j=0}^{\kappa+1} \varepsilon^j h_j \right), \nabla \left( \sum_{j=0}^{\kappa +1} \varepsilon^j f_j + \sum_{j=0}^{\kappa +1 } \varepsilon^{\delta-j} \pi^* \phi_j \right) \right\rangle,
$$ 
Thus we have the required expansion with $f_{\kappa+1}$ and $h_{\kappa +1}$ as above and $\phi_{\kappa +1} = 0$.

It remains to show that there are no horizontal terms up to order $\varepsilon^{\delta}$, apart from our original one at order exactly $\varepsilon^{\delta}$. This is again a consequence of Proposition \ref{prop:linearisation}, since the linearisation only hits horizontal terms at order $\varepsilon^{2\delta}$.

We now continue for $\kappa \geq \delta$. In this case, there will be horizontal error terms, too, and we begin by correcting these. Suppose 
$$
S \left( \Omega_{\varepsilon, \kappa} \right) - h_{\varepsilon, \kappa} = l_{\kappa +1} \varepsilon^{\kappa + 1} + \psi_{\kappa+1} \varepsilon^{\kappa+1} + O(\varepsilon^{\kappa + 2}),
$$
where $l_{\kappa+1} \in C^{\infty}_0 (\scX)$ and $\psi_{\kappa+1} \in C^{\infty} (\mathbb{P}^1)$. By the mapping properties of the Lichnerowicz operator on $\mathbb{P}^1$, there exists $\phi_{\kappa+1} \in C^{\infty} (\mathbb{P}^1)$ and a holomorphy potential $h^1_{\kappa+1}$ on $\mathbb{P}^1$, with respect $\omega_{FS}$, such that 
$$
\mathcal{D}^*_{\omega_{FS}} \mathcal{D}_{\omega_{FS}} (\phi_{\kappa +1 })  = \psi_{\kappa+1} - h^1_{\kappa+1}.
$$

By Proposition \ref{prop:linearisation} and Lemma \ref{lem:linhorizontal}, we therefore have that $\Omega_{\varepsilon, \kappa} + \ddb \left( \varepsilon^{\kappa +1 - 2 \delta} \right)$ is K\"ahler and 
\begin{align*}
S \left( \Omega_{\varepsilon, \kappa} + \ddb \left( \varepsilon^{\kappa+1-2\delta} \phi_{\kappa+1} \right) \right)=&S \left( \Omega_{\varepsilon, \kappa}  \right) - \varepsilon^{\kappa + 1} \psi_{\kappa+1} + \varepsilon^{\kappa+1} h^1_{\kappa+1}.
\end{align*}
Note that Lemma \ref{lem:linhorizontal} applies because $\kappa \geq \delta$, so $\kappa +1 - 2 \delta > - \delta$. 

We now invoke Lemma \ref{lem:holpots}, which allows us to compare $h_{\kappa + 1}^1$ to an actual holomorphy potential for $\Omega_{\varepsilon}$ on $\scX$.  We have that $h^1_{\kappa+1} = \upsilon h_{FS} + c$ for some $\upsilon$ constants $\upsilon$ and $c$. If we then write
$$
h^2_{\kappa +1} = \upsilon h_{\varepsilon} = \upsilon \left( \chi (t/\varepsilon) \cdot  h_0  + \lambda \varepsilon^{-\delta} h_{FS} \right)
$$
for the average zero potential for the corresponding multiple of the generator of the circle action on $\scX$, we then have that
\begin{align*}
S \left( \Omega_{\varepsilon, \kappa} + \ddb \left( \varepsilon^{\kappa+1-2\delta} \phi_{\kappa+1} \right) \right)=& \varepsilon^{\kappa+1-\delta} h^2_{\kappa+1} + \varepsilon^{\kappa+1} c + h_{\varepsilon, \kappa} \\
&+  \left\langle \nabla h_{\varepsilon,\kappa} , \nabla \left( \sum_{j=0}^{\kappa } \varepsilon^j f_j + \sum_{j=1}^{\kappa +1 } \varepsilon^{j-\delta} \pi^* \phi_j \right) \right\rangle\\
& + \varepsilon^{\kappa+1} l_{\kappa+1} + O(\varepsilon^{\kappa +2}).
\end{align*}

We now define 
$$
\widetilde{h}_{\varepsilon, \kappa +1} = h_{\varepsilon, \kappa} + \varepsilon^{\kappa+1} c + \varepsilon^{\kappa+1-\delta} h^2_{\kappa+1},
$$
so we have added an $O(\varepsilon^{\kappa+1})$ constant term, but also \emph{altered} the $\varepsilon^{\kappa+1-\delta}$-term, compared to $h_{\varepsilon, \kappa}$. If we then compare 
$$
S \left( \Omega_{\varepsilon, \kappa} + \ddb \left( \varepsilon^{\kappa+1-2\delta} \phi_{\kappa+1} \right) \right)
$$
to 
$$
\widetilde{h}_{\varepsilon,\kappa +1}  +  \left\langle \nabla \widetilde{h}_{\varepsilon,\kappa +1} , \nabla \left( \sum_{j=0}^{\kappa } \varepsilon^j f_j + \sum_{j=1}^{\kappa +1 } \varepsilon^{j-\delta} \pi^* \phi_j \right) \right\rangle
$$
we have no horizontal term to up and including order $\varepsilon^{\kappa+1}$: the only horizontal term that could appear comes from
\begin{align*}
\left\langle \varepsilon^{\kappa+1-\delta} h^2_{\kappa+1} , \nabla \left( \sum_{j=1}^{\kappa +1 } \varepsilon^{j-\delta} \pi^* \phi_j \right) \right\rangle 
\end{align*}
whose horizontal component to leading order is
\begin{align*}
\left\langle \varepsilon^{\kappa+1} h^1_{\kappa+1} , \nabla \left( \sum_{j=1}^{\kappa +1 } \varepsilon^{j-\delta} \pi^* \phi_j \right) \right\rangle_{\varepsilon^{-\delta} \omega_{FS}} &= \varepsilon^{ (\kappa +1) + (1 - \delta) + \delta } \left\langle h_{\kappa+1}^1, \phi_1 \right\rangle_{\omega_{FS}} \\
&= \varepsilon^{ \kappa +2} \left\langle h_{\kappa+1}^1, \phi_1 \right\rangle_{\omega_{FS}}.
\end{align*}
So the horizontal terms appearing in the difference of the two terms are at least $O(\varepsilon^{\kappa+2})$, as claimed.

On the other hand, we have introduced new vertical terms at smaller orders. These come from accounting for the inner product of 
$$
\nabla \left( \sum_{j=0}^{\kappa } \varepsilon^j f_j + \sum_{j=1}^{\kappa +1 } \varepsilon^{j-\delta} \pi^* \phi_j \right) 
$$
with the gradient of $\varepsilon^{\kappa+1-\delta} h^2_{\kappa+1}$. The leading order contribution is therefore the inner product with the gradient of $\varepsilon^{\kappa + 1 - \delta} \cdot \upsilon \cdot \chi (t/\varepsilon) \cdot h_0$. Since $f_0 =0$, this means we get new vertical contributions to all orders starting from order $\varepsilon^{\kappa+2-\delta}$. The upshot is that we have an expansion
\begin{align*}
S \left( \Omega_{\varepsilon, \kappa} + \ddb \left( \varepsilon^{\kappa+1-2\delta} \phi_{\kappa+1} \right) \right)=& \widetilde{h}_{\varepsilon,\kappa +1}  +  \left\langle \nabla \widetilde{h}_{\varepsilon,\kappa +1} , \nabla \left( \sum_{j=0}^{\kappa } \varepsilon^j f_j + \sum_{j=1}^{\kappa +1 } \varepsilon^{j-\delta} \pi^* \phi_j \right) \right\rangle \\
& + \sum_{j=0}^{\delta} \varepsilon^{\kappa+2+ j - \delta} l_j^{\kappa+1} + O(\varepsilon^{\kappa +2}),
\end{align*}
where all the $l_j^{\kappa+1} \in C^{\infty}_0 (\scX)$ are vertical. 

Following exactly the same argument as for the case when $\kappa < \delta$, we can remove these errors \emph{without reintroducing horizontal terms} until order $\varepsilon^{\kappa+2}$. This relies on the fact that when using vertical functions added at order $\varepsilon^i$, we are not introducing any horizontal terms until order $\varepsilon^{i+\delta}$. The upshot is that by altering the $f_j$'s in the previous steps and adding a suitable function $f_{\kappa+1}$ at order $\varepsilon^{\kappa+1}$, we get an expansion of the required type. This completes the proof.
\end{proof}

\subsection{Perturbing to a solution}
We are now ready to perturb the approximate solution constructed above to a genuine solution. 
\begin{proposition}
\label{prop:inversebound}
Suppose $\delta>1$. Let $\Psi_{\varepsilon, \kappa} : C^{k+4,\alpha} \times \overline{\mathfrak{h}} \to C^{k,\alpha}$ denote the operator
$$
\Psi_{\varepsilon, \kappa} (\phi, h) = L_{\varepsilon, \kappa} (\phi) - h_{\varepsilon, \kappa} - \frac{1}{2} \langle \nabla_{\Omega_{\varepsilon, \kappa}} (h_{\varepsilon, \kappa}), \nabla \phi \rangle,
$$
where $L_{\varepsilon, \kappa}$ is the linearisation of the scalar curvature operator at $\Omega_{\varepsilon, \kappa}$. Then $\Psi_{\varepsilon, \kappa}$ admits a right-inverse $Q_{\varepsilon, \kappa}$ satisfying
$$
\| Q_{\varepsilon, \kappa} \| \leq C \varepsilon^{- 2 \delta}.
$$
\end{proposition}
\begin{proof}
The proof is very similar to analogous results for other fibrations, see \cite[Lemma 6.5, 6.6, 6.7]{fine04}  or \cite[Lemma 6.5]{dervansektnan20}, and we omit the details. The precise rate $\varepsilon^{-2 \delta}$ of the bound comes from Proposition \ref{prop:linearisation}, which shows that $\Psi_{\varepsilon, \kappa}$ is surjective modulo $\overline{\mathfrak{h}}$ at order $\varepsilon^{2 \delta}$.
\end{proof}

We are now ready to perturb to a genuine solution of the extremal equation. The key will be the following Quantitative Inverse Function Theorem.
\begin{theorem}
\label{thm:implicit}
Suppose $\Phi : V \to W$ is a differentiable map of Banach spaces, with surjective differential at $0 \in V$. Let $\Psi$ be a right inverse for $D\Phi_0$. Let 
\begin{itemize}
\item $r'$ be the radius of the closed ball in $V$ where $\Phi - d \Phi$ is Lipschitz, with Lipschitz constant $\frac{1}{2 \| \Psi \|}$;
\item $r = \frac{1}{2 \| \Psi \|}$.
\end{itemize}
Then for all $w \in W$ such that $\| w - \Phi (0) \| < r$, there exists a $v \in V$ with $\| v \| < r'$ such that $\Phi (v) = w.$
\end{theorem}

We will this to the operator 
$$
\Phi_{\varepsilon, \kappa} : C^{k+4,\alpha} (\scX)  \times \overline{\mathfrak{h}} \to C^{k,\alpha} (\scX)
$$
given by
\begin{align}
\label{eq:theeqn}
(f,h) \mapsto S(\Omega_{\varepsilon, \kappa} + \ddb f) - h_{\varepsilon, \kappa} - \frac{1}{2} \langle \nabla_{\Omega_{\varepsilon, \kappa}} (h_{\varepsilon, \kappa}), \nabla \phi \rangle .
\end{align}
Notice that the linearisation of $\Phi_{\varepsilon, \kappa}$ is the operator $\Psi_{\varepsilon, \kappa}$ in Proposition \ref{prop:inversebound}. In order to apply Theorem \ref{thm:implicit}, we need to a definite region on which $\Phi= \Phi_{\varepsilon, \kappa} $ is Lipschitz of Lipschitz constant $\frac{1}{2 \| \Psi \|}$. This is provided by the following lemma.
\begin{lemma}
\label{lem:lipschitz}
Let $\mathcal{N}_{\varepsilon, \kappa} = \Phi_{\varepsilon, \kappa} - d \Phi_{\varepsilon, \kappa}$, where $\Phi_{\varepsilon, \kappa} : C^{k+4,\alpha} \times \mathfrak{h} \to C^{k,\alpha}$ is the operator given by Equation \eqref{eq:theeqn}. There are constant $c,C >0$, such that for all $\varepsilon >0$ sufficiently small, if $f_i \in C^{k+4,\alpha} \times \mathfrak{h}$ for $i=1,2$ satisfy $\| f_i \| \leq c$, then 
$$
\| \mathcal{N}_{\varepsilon, \kappa} (f_1) - \mathcal{N}_{\varepsilon, \kappa} (f_2) \| \leq C \left( \| f_1 \| + \| f_2 \| \right) \| f_1 - f_2 \|.
$$
\end{lemma}
The proof of Lemma \ref{lem:lipschitz} is a direct consequence of the Mean Value Theorem, using that the $\Psi_{\varepsilon, \kappa}$ are bounded, independently of $\varepsilon$, and the following result, which is a global version of Proposition \ref{prop:linearisationchange}, and whose proof is similar.
\begin{proposition}
\label{prop:linearisationchange2}
For each $k,\alpha$, there exists $c, C>0$ such that if $\| \phi \|_{C^{k+4,\alpha}} \leq c$, then 
$$
\| L_{\Omega_{\varepsilon} + \ddb \phi} (f) - L_{\Omega_{\varepsilon}}  (f) \|_{C^{k,\alpha}} \leq C \| \phi \|_{C^{k+4,\alpha}} \| f \|_{C^{k+4,\alpha}}.
$$
\end{proposition}

We are now ready to prove our main result. The remaining argument follows exactly similar arguments in e.g. \cite{fine04,dervansektnan21}. The statement below is just rephrasing Theorem \ref{thm:main}, using the parameter $\varepsilon$ instead of $k$.
\begin{theorem}
\label{thm:mainepsilon}
For all $\varepsilon >0$ sufficiently small, there exists an extremal metric $\widetilde{\Omega}_{\varepsilon} \in [\Omega_{\varepsilon}]$. 
\end{theorem}
\begin{proof}
Lemma \ref{lem:lipschitz} implies that $\mathcal{N}_{\varepsilon,\kappa}$ is Lipschitz on the ball of radius $\varrho$, with Lipschitz constant $c \varrho$, for all $\varrho$ sufficiently small. Letting $r'$ be as in Theorem \ref{thm:implicit} for the extremal operator, we therefore have that $r' \geq c \varepsilon^{2 \delta}$, for some constant $c>0$, by Proposition \ref{prop:inversebound}. Letting $r= \frac{r'}{2 \| Q_{\varepsilon, \kappa} \|}$, we then have that 
$$
r \geq C \varepsilon^{4\delta},
$$
for some constant $C> 0$. In particular, Theorem \ref{thm:implicit} implies that we can perturb the scalar curvature, up to a holomorphy potential, to anything in the ball of radius $C  \varepsilon^{4 \delta}$ about $S(\Omega_{\varepsilon,\kappa})$. Since $C$ depends only on $\kappa,$ not $\varepsilon$, we see that if $\kappa$ is chosen sufficiently large (we need $\kappa > 4 \delta$), then we can solve the extremal equation, since $S(\Omega_{\varepsilon,\kappa})$ is extremal to order $\varepsilon^{\kappa+1}$.
\end{proof}

\begin{remark}
\label{rem:closesolns}
In fact, as used e.g. in \cite[Remark 3.8]{dervansektnan21}, the actual solutions can be ensured to be an $O(\varepsilon^{\kappa - 2 \delta})$ perturbation of $\Omega_{\varepsilon, \kappa}$. In particular, we can, for any $\kappa$, choose a $\kappa'>\kappa$ such that the actual solutions produced from $\Omega_{\varepsilon,\kappa'}$ agree with $\Omega_{\varepsilon, \kappa'}$ to order $\varepsilon^{\kappa}$. In particular, for any desired $k$ and $\alpha$, and for all $\varepsilon$ sufficiently small, the solution will agree with 
$$
\Omega_{\varepsilon} = \omega_{\varepsilon} + \lambda \varepsilon^{-\delta} \pi^* \omega_{FS}.
$$ 
in $C^{k,\alpha}$, up to terms that are decaying in $\varepsilon$ in both the horizontal and vertical direction. 
\end{remark}

\subsection{Larger automorphism groups}

In this section, we detail how to extend the above result to a more general setting, where we instead allow a difference of a $\mathbb{C}^*$ in the automorphism groups of the general fibre and the central fibre. More precisely, we will assume the following
\begin{definition}
\label{def:automassumption}
Suppose that a maximal torus $T_0$ for $\scX_0$ acts on the Kuranishi slice in such a way that the stabiliser $T$ of the point corresponding to $X$ is a  torus inside the reduced automorphism group of $X$ and that $T_0 = T \times \mathbb{C}^*$, where the $\mathbb{C}^*$ component is the one generating the test configuration. We then say that the automorphism group of $(X,L)$ has a \emph{$\mathbb{C}^*$ discrepancy against the automorphism group of the cscK central fibre $(\scX_0, \mathcal{L}_0)$}.
\end{definition}
A different way of thinking about this, which we will prove below, is that this condition means that there is an action of the full maximal torus $T$ of the central fibre, on the total space of the degeneration. In general, the test configuration is only preserved by a $\mathbb{C}^*$ subgroup of $T$, but the $\mathbb{C}^*$ discrepancy condition says that the maximal torus of the central fibre lifts to a torus inside the automorphism group of the total space of the test configuration.

Under the above condition, we can extend our main result.
\begin{theorem}
\label{thm:mainepsilon2}
Suppose the automorphism group of $(X,L)$ has a $\mathbb{C}^*$ discrepancy against the automorphism group of the cscK central fibre $(\scX_0, \mathcal{L}_0)$. Then, for all $\varepsilon >0$ sufficiently small, there exists an extremal metric $\widetilde{\Omega}_{\varepsilon} \in [\Omega_{\varepsilon}]$. 
\end{theorem}

The main new point is to understand the holomorphy potentials on $\scX$ that arise from vector fields on $\scX_0$ that extend to the general fibres, similarly to Lemma \ref{lem:holpots} for the ones that do not have this property. 

We first explain how the action of $T$ extends to the test configuration. The manifold $\scX$ can be seen as the gluing of $U_1 = M \times \Delta$ and $U_2 = M \times \mathbb{C}$ via 
$$
(m, t) \sim (f_t (m), \frac{1}{t}),
$$
where $f_t : M \to M$ is the diffeomorphism such that $J_t = f_t^* (J_1)$, see Section \ref{sec:metricsextension}. Now, the action of $T$ acts trivially on the second component in $U_1$ and commutes with the $\mathbb{C}^*$ action. We therefore have that, if $\lambda \cdot $ denotes the action of some element $\lambda \in T$, then
\begin{align*}
\lambda \cdot (m, t) =& (\lambda \cdot m, t) \\
\sim & ( f_t (\lambda \cdot m), \frac{1}{t}) \\
=& ( \lambda \cdot f_t (m),  \frac{1}{t})\\
=& \lambda \cdot ( f_t (m),  \frac{1}{t}).
\end{align*}
Thus, we can define the action to be the $T$-action on $M$ in either chart, giving a well-defined $T$-action on $\scX$. As the action is holomorphic in the first chart from Kuranishi theory and holomorphic in the second chart by the assumption that the $T$ action is holomorphic on $(M, J_1)$, we have that this action is holomorphic on $\scX$.

We can now consider the holomorphy potentials for this action.
\begin{lemma}
\label{lem:holpots2}
Let $h$ be a holomorphy potential in $\scX_0$ with respect to $\omega$, whose associated vector field lies in the image of the Lie algebra of $T$ inside the vector fields on $M$. Then a holomorphy potential for the corresponding vector field on $\scX$ with respect to $\Omega_{\varepsilon}$ is given by
$$
h_{\varepsilon} = h +  \frac{1}{2}  \langle \nabla h , \nabla \psi_{\varepsilon} \rangle,
$$
where 
$$
\psi_{\varepsilon} = \chi \left( \frac{|t|}{\varepsilon} \right) \phi_t + (1-\chi \left( \frac{|t|}{\varepsilon} \right) ) \phi_{\varepsilon} 
$$
is the potential of $\omega_{\varepsilon}$ with respect to $\omega$. 
\end{lemma}
\begin{proof} 
On $X=(M,J_1)$, $\omega$ is a K\"ahler form, and so on $U_2 = X \times \mathbb{C}$, $h$ is a holomorphy potential for the vector field with respect to $\omega$. Thus, to obtain the expression for the holomorphy potential with respect to $\omega_{\varepsilon}$ fibrewise, we can simply use the standard expression for the change in holomorphy potentials
$$
h_{\varepsilon} = h + \frac{1}{2}  \langle \nabla h , \nabla \psi_{\varepsilon} \rangle,
$$
where $\psi_{\varepsilon} $ is the potential of $\omega_{\varepsilon}$ with respect to $\omega$. This is also the potential with respect to $\Omega_{\varepsilon}$, since the $T$-action covers the trivial action on $\mathbb{P}^1$.
\end{proof}

With this in place, we can now explain how to adapt the argument to extend the results of Theorem \ref{thm:mainepsilon} to Theorem \ref{thm:mainepsilon2}. The linearised operator on the central fibre now has a larger cokernel, meaning that in Proposition \ref{prop:linearisation}, we now got a bigger cokernel in the leading order term
$$
P_{\varepsilon} = L_{\omega} + O(\varepsilon).
$$
But this cokernel $\overline{\mathfrak{h}}$ still consists of the holomorphy potentials on the total space $\scX$, the only change is that there are more of them. One can then in a very similar manner as the previous case show that the operator $\Psi_{\varepsilon, \kappa}$ given by
$$
(\phi, h) \mapsto L_{\varepsilon, \kappa} (\phi) - h_{\varepsilon, \kappa} - \frac{1}{2} \langle \nabla_{\Omega_{\varepsilon, \kappa}} (h_{\varepsilon, \kappa}), \nabla \phi \rangle,
$$
is surjective, and the bounds of Proposition \ref{prop:inversebound} for the right inverse go through in exactly the same manner. The only difference is that the space $\overline{\mathfrak{h}}$ is now larger.

When adapting the argument of Proposition \ref{prop:approxsoln}, the properties of the horizontal part remains the same, and the argument is exactly the same there. When taking care of the vertical errors, there is now a bigger cokernel in the leading order term in the expansion of the linearised operator. But  these additional type of cokernel elements are precisely the functions of the form given by Lemma \ref{lem:holpots2}, which are global holomorphy potentials on $\scX$. These functions are still $O(1)$ in $\varepsilon$, and do not introduce any horizontal terms at an earlier level than the holomorphy potential of Lemma \ref{lem:holpots} does. Thus we can remove vertical errors without introducing horizontal terms at too early an order in $\varepsilon$. 

The rest of the proof now goes through in the same manner. We can create an approximate solution that solves the extremal equation to any order we would like, and the operator norm of the inverse is of the same order in $\varepsilon$ as the previous case. Thus we can apply the exact same argument as in the proof of Theorem \ref{thm:mainepsilon}, to create an actual extremal metric using the quantitative implicit function theorem, Theorem \ref{thm:implicit}. This completes the proof of Theorem \ref{thm:mainepsilon2}.

\begin{remark}
\label{rem:outlook}
Treating $\delta$ as a parameter is not strictly speaking necessary in our construction. One could e.g. choose $\delta = 2$. On the other hand, under more general assumptions on the relationship between the automorphism group of the central fibre and the general fibre, constructions like ours should be possible. This is the main reason we have treated $\delta$ as a parameter. 

These constructions should be obstructed, as there is no longer an action of the full maximal torus $T_0$ of the central fibre, on the total space of the test configuration. It is likely that there one needs a more specific choice of $\delta$ to make the construction work. For example, if one would need to pick $\delta$ to be minimal, i.e. $\delta = 1$, in the construction, one would potentially get a different equation than the cscK equation for the metric on $\mathbb{P}^1$ by looking at the horizontal term at this order -- this would likely be a twisted cscK equation, with twisting term coming from the pushforward of certain terms from the total space of the test configuration. This would also lead to obstructions at order $\varepsilon$ in the vertical direction, that one cannot deal with using the linearised operator.

In fact, at the time of writing, a general framework for attacking these questions on fibrations with only semistable fibres is being developed by Annamaria Ortu (\cite{ortu21}). Her work studies the analogue of the optimal symplectic connection equation in this setting, and constructs extremal metrics on the total space of such fibrations, when this equation and a suitable equation on the base can be solved. This requires one to see the total space of the fibration as a deformation of a cscK fibration, which is very different from the approach taken in our specific case. 

It seems possible that our work could fit into her framework. The key to this would be to show that the test configuration we consider can be seen as a deformation of a product test configuration over $\mathbb{P}^1$ for $(\scX_0, \mathcal{L}_0)$. The likely choice is to use the product test configuration produced from the $\mathbb{C}^*$ action on $\scX_0$, that we used to define the $\mathbb{C}^*$ action on the initial local test configuration we have considered. This means that we would \emph{not} see this as a deformation of $\scX_0 \times \mathbb{P}^1$, in general. Showing that indeed this is the case may be the better avenue to pursue in order to extend our results to cases with more general automorphism groups.
\end{remark}

\section{Metric limits of the adiabatic extremal metrics}
\label{sec:GH}

It is interesting to analyze from the metric viewpoint what happens to the extremal metrics $\widetilde{\Omega}_{\varepsilon}$ when $\varepsilon\rightarrow 0$. 

\begin{proposition}
	For any point $p\in \mathcal{X}$, the metrics $\widetilde{\Omega}_{\varepsilon}$ converge as $\varepsilon\rightarrow 0$ in the pointed Gromov-Hausdorff sense to the product $X_0 \times \mathbb{R}^2$, with $X_0$ equipped with its constant scalar curvature metric and flat  $\mathbb{R}^2$.
\end{proposition}

\begin{remark}
This shows that a phenomenon of jumping of complex structures in metric limits happens in the strictly K-semistable fibres $X_t$ as $t\neq 0$. Indeed, instead of converging to some K\"ahler metric on the product  $X_t \times \mathbb{R}^2$, as $\varepsilon$ goes to $0$, these extremal metrics $\widetilde{\Omega}_{\varepsilon}$ are resembling in all the vertical directions more and more  the constant scalar curvature K\"ahler metric at the central fibre of the test configuration.
\end{remark}

\begin{proof} The result follows by observing that this limit behavior happens for the background model metric $\Omega_{\varepsilon}$. At any fixed fibre as $\varepsilon\rightarrow 0$, the vertical part of  $\Omega_{\varepsilon}$, namely $\omega_\varepsilon$, gets metrically  closer and closer to the cscK manifold $(X_0, \omega_0)$. Indeed, the vertical background metrics on the fibres for $ | t |\geq \varepsilon$ get smoothly closer to the cscK metric, thanks to their construction as smooth equivariant deformations in (\cite{szekelyhidi10}).  Moreover, the horizontal part gets more and more dominated by a large multiple of the Fubini-Study metric. Then the analytic estimates in the perturbative analysis, see Remark \ref{rem:closesolns}, show that the two metrics tensors corresponding to  $\Omega_\varepsilon$ and $\widetilde{\Omega}_{\varepsilon}$ are point-wise getting closer as $\varepsilon\rightarrow 0$, so the difference of the distances measured with the two metrics goes to zero. Hence, passing to the limit, both the model metrics $\Omega_{\varepsilon}$ and the extremal metrics $\widetilde{\Omega}_{\varepsilon}$ converge in the pointed Gromov-Hausdorff sense to the same space $X_0 \times \mathbb{R}^2$, with $X_0$ equipped with its constant scalar curvature metric and flat  $\mathbb{R}^2$, as claimed.
\end{proof}

Finally, note that if instead we would have rescaled the metrics so that the horizontal direction remains of fixed diameter (while the fibres shrink to zero) the metrics would converge to the Fubini-Study metric on the base $\mathbb{P}^1$.

\section{Examples}
\label{sec:examples}

In this final section we show that our construction can be used to produce many new extremal metrics. In order to apply the construction, we need to find a destabilising test configuration to a smooth cscK central fibre, for a strictly semistable manifold. Moreover, we need the $\mathbb{C}^*$ discrepancy condition to hold. 

All the examples we give here will come from explicit Fano threefolds. We consider some different families of the Mori--Mukai classification of Fano threefolds in turn. This relies heavily on recent work on Fano threefolds by many authors. We will use the book \cite{fanothreefolds} as our chief reference for the results relevant to our construction.

\subsection{Family 2.24} Here we will consider a special member of the family 2.24, following \cite[Corollary 4.7.7]{fanothreefolds}.
\begin{lemma}
\label{lem:2.24}
Let $X$ be the Fano threefold in the family 2.24 of the Mori--Mukai list given by equation
$$
(vw + u^2)x + v^2 y + w^2 z = 0,
$$
in $\mathbb{P}^2 \times \mathbb{P}^2$ with homogenous coordinates $([x,y,z],[u,v,w])$. Then there exists a test configuration $(\mathcal{X}, \mathcal{L})$ for $(X, - K_X)$ for which the construction of Theorem \ref{thm:main} applies.
\end{lemma}
\begin{proof}
Consider the family $X_s$ with $s \in \mathbb{C}$ given by 
$$
(s vw + u^2)x + v^2 y + w^2 z = 0.
$$
By rescaling the variables $u,v,w,x,y,z$ one sees that $X_s$ is isomorphic to $X=X_1$ when $s \neq 0$. On the other hand, when $s=0$, we obtain the Fano manifold $X_0$ given by
$$
 u^2 x + v^2 y + w^2 z = 0.
$$ 
By \cite[Lemma 4.7.6]{fanothreefolds}, $X_0$ is $K$-polystable and so the family $X_s$ is a destabilising test configuration, showing that $X$ is strictly $K$-semistable. Moreover, by \cite[Lemma 4.7.5]{fanothreefolds} (which is a corollary of \cite[Lemma 10.2]{cps19}), $X_0$ has automorphism group $(\mathbb{C}^*)^2$, while $X$ has automorphism group $\mathbb{C}^*$. Thus the $\mathbb{C}^*$ discrepancy condition is satisfied.

Versality in the Kuranishi theory then implies that there exists a test configuration $(\mathcal{X}, \mathcal{L})$ for $(X, -K_X)$ in the Kuranishi family, whose central fibre is $(X_0, -K_{X_0})$. The fact that central fibre of this potentially different test configuration to the one we have explicitly produced is $(X_0, - K_{X_0})$ follows by uniqueness of cscK degenerations (\cite[Theorem 1.3]{chensun14}). Since the $\mathbb{C}^*$ discrepancy condition is satisfied, Theorem \ref{thm:main} applies to construct extremal metrics on $(\mathcal{X}, \mathcal{L}+ \mathcal{O}(d))$ for all sufficiently large $d$.
\end{proof}
\begin{remark}
In \cite[Corollary 4.7.7]{fanothreefolds}, it is also shown that a different member $Y$ of the family 2.24 is strictly K-semistable, with central fibre the same polystable $X_0$ as above. However, this $Y$ has discrete automorphism group and so does \emph{not} satisfy the $\mathbb{C}^*$ discrepancy condition. This family could potentially be a non-trivial example to investigate in the setting developed by Ortu (\cite{ortu21}), see Remark \ref{rem:outlook}.
\end{remark}

\subsection{Family 3.10} We can also do the construction in families, applying it to a one-parameter family of strictly K-semistable manifolds in the family 3.10. Let $X_c$ for $c \neq \pm 1$ be the Fano threefold in the family 3.10 of the Mori--Mukai list given as follows. Endow $\mathbb{P}^4$ with homogeneous coordinates $[v,w,x,y,z]$. Let $C_1$ and $C_2$ be the curves in $\mathbb{P}^4$ given by
\begin{align*}
C_1 =& \{ 0 = w^2 + z v = y = x  \}, \\
C_2 =& \{ 0 = w^2 + xy = v = z  \}.
\end{align*} 
These curves are disjoint and are both contained in the quadric surfaces
$$
Q_c = \{ w^2 + xy + zv + c( xv + yz ) + xz = 0 \} 
$$
for any $c$. Let $X_c = \Bl_{C_1, C_2} Q_c$ be the blowup of $Q_c$ in the two curves $C_1$ and $C_2$. For all $c$, the reduced automorphism group of $X_c$ is trivial (\cite[Lemma 5.9]{cps19}).
\begin{proposition}
\label{prop:3.10}
For all $c \neq 0, \pm 1$, there exists a test configuration $(\mathcal{X}_c, \mathcal{L}_c)$ for the member $(X_c, - K_{X_c})$ of the family 3.10 for which the construction of Theorem \ref{thm:main} applies.
\end{proposition}
Note that the condition $c \neq \pm 1$ is made just to ensure that we have a smooth threefold. The case $c=0$ is excluded because the $\mathbb{C}^*$ discrepancy condition is not satisfied in this case.
\begin{proof}
This follows \cite[Corollary 5.17.7]{fanothreefolds}, which gives an explicit test configuration for $X_c$, given as follows. Let
$$
Q_{c,s} = \{ w^2 + xy + zv + c( xv + yz ) + s xz = 0 \} .
$$
All the $Q_{c,s}$ contain $C_1$ and $C_2$, and so we can define $X_{c,s} =  \Bl_{C_1, C_2} Q_{c,s}$. Then for all non-zero $s$, $X_{c,s}$ is isomorphic to $X_c$. On the other hand, when $s=0$, $Y_c = X_{c,0}$ is K-polystable (\cite[Lemma 5.17.6]{fanothreefolds}). This gives a destabilising test configuration for $(X_s, - K_{X_s}).$ Moreover, provided $c \neq 0$, by \cite[Lemma 5.9]{cps19}, the automorphism group of $Y_c$ is $\mathbb{C}^*$, showing the construction can be applied to all of these $X_c$. 
\end{proof}
For the remaining case $c=0$, the automorphism group of $Y_c$ is $(\mathbb{C}^*)^2$, and so the $\mathbb{C}^*$ discrepancy condition is not satisfied in this case.

\subsection{Family 4.13} Finally, we consider the family 4.13, where we again focus on one special member of the family.
\begin{lemma}
\label{lem:4.13}
Let $X$ be the Fano threefold in the family 4.13 of the Mori--Mukai list given as the blowup of $\mathbb{P}^1 \times \mathbb{P}^1 \times \mathbb{P}^1$, with homogenous coordinates $([x,y],[u,v],[p,q])$, in the curve given by the two equations
$$
xv - y u = x^3 p + y^3 q + x y^2 p = 0.
$$
Then there exists a test configuration $(\mathcal{X}, \mathcal{L})$ for $(X, - K_X)$ for which the construction of Theorem \ref{thm:main} applies.
\end{lemma}
\begin{proof}
The proof is very similar to that of Lemma \ref{lem:2.24}. By \cite[Corollary 5.22.3]{fanothreefolds}, $X$ admits a test configuration degenerating $X$ to the Fano threefold $X_0$ given as the blowup of $\mathbb{P}^1 \times \mathbb{P}^1 \times \mathbb{P}^1$ in the curve
$$
xv - y u = x^3 p + y^3 q = 0.
$$
This central fibre is K-polystable by \cite[Theorem 5.22.7]{fanothreefolds}, and has automorphism group $\mathbb{C}^*$, while $X$ has trivial automorphism group. It follows that the construction applies to a test configuration for $(X, -K_X)$.
\end{proof}

\subsection{Summary of examples}
We summarise all the special cases we have considered in the following theorem.
\begin{theorem}
\label{prop:examples}
The following families from the Mori--Mukai list of Fano threefolds produce at least one test configuration for a strictly K-semistable manifold to which the construction of Theorem \ref{thm:main} applies: 
\begin{itemize}
\item 2.24
\item 3.10
\item 4.13
\end{itemize}
\end{theorem}

\begin{remark}
The above result says that in each of the cases mentioned, we have found at least \emph{one} member, to which we can apply our construction. In some cases, we have checked families rather than a single member, but we have not checked every possible member of the family in order to get a complete classification. The precise members of the given families that we have checked are given in the relevant results prior to the theorem. We certainly think there are other members of these and other families from which we can obtain further examples. 
\end{remark}

\begin{remark} Not all the families will give test configurations to which we can apply the construction, even if there are strictly K-semistable members of these families. For example, there are some strictly K-semistable members of the family 2.26, but they degenerate to a \emph{singular} K-stable Fano variety. There are no K-polystable members of this family (\cite[Section 5.10]{fanothreefolds}), and so our construction cannot be applied to any members of this family.
\end{remark}

\bibliography{degenerations}

\providecommand{\bysame}{\leavevmode\hbox to3em{\hrulefill}\thinspace}
\providecommand{\MR}{\relax\ifhmode\unskip\space\fi MR }
\providecommand{\MRhref}[2]{%
  \href{http://www.ams.org/mathscinet-getitem?mr=#1}{#2}
}
\providecommand{\href}[2]{#2}
\begin{thebibliography}{10}

\bibitem{acgtf04}
Vestislav Apostolov, David M.~J. Calderbank, Paul Gauduchon, and Christina~W.
  T{\o}nnesen-Friedman, \emph{Hamiltonian 2-forms in {K}\"{a}hler geometry.
  {II}. {G}lobal classification}, J. Differential Geom. \textbf{68} (2004),
  no.~2, 277--345. \MR{2144249}

\bibitem{fanothreefolds}
Carolina Araujo, Ana-Maria Castravet, Ivan Cheltsov, Kento Fujita, Anne-Sophie
  Kaloghiros, Jesus Martinez-Garcia, Constantin Shramov, Hendrik Süss, and
  Nivedita Viswanathan, \emph{{The Calabi Problem for Fano threefolds}}, MPIM
  Preprint Series \textbf{31} (2021).

\bibitem{arezzopacard06}
Claudio Arezzo and Frank Pacard, \emph{Blowing up and desingularizing constant
  scalar curvature {K}\"ahler manifolds}, Acta Math. \textbf{196} (2006),
  no.~2, 179--228. \MR{2275832 (2007i:32018)}

\bibitem{arezzopacard09}
\bysame, \emph{Blowing up {K}\"ahler manifolds with constant scalar curvature.
  {II}}, Ann. of Math. (2) \textbf{170} (2009), no.~2, 685--738. \MR{2552105
  (2010m:32025)}

\bibitem{arezzopacardsinger11}
Claudio Arezzo, Frank Pacard, and Michael Singer, \emph{Extremal metrics on
  blowups}, Duke Math. J. \textbf{157} (2011), no.~1, 1--51. \MR{2783927}

\bibitem{bronnle15}
Till Br\"{o}nnle, \emph{Extremal {K}\"{a}hler metrics on projectivized vector
  bundles}, Duke Math. J. \textbf{164} (2015), no.~2, 195--233. \MR{3306554}

\bibitem{calabi82}
Eugenio Calabi, \emph{Extremal {K}\"ahler metrics}, Seminar on {D}ifferential
  {G}eometry, Ann. of Math. Stud., vol. 102, Princeton Univ. Press, Princeton,
  N.J., 1982, pp.~259--290. \MR{645743 (83i:53088)}

\bibitem{cannasdasilva01}
Ana Cannas~da Silva, \emph{Lectures on symplectic geometry}, Lecture Notes in
  Mathematics, vol. 1764, Springer-Verlag, Berlin, 2001. \MR{1853077}

\bibitem{cps19}
I.~A. Cheltsov, V.~V. Przyjalkowski, and C.~A. Shramov, \emph{Fano threefolds
  with infinite automorphism groups}, Izv. Ross. Akad. Nauk Ser. Mat.
  \textbf{83} (2019), no.~4, 226--280. \MR{3985696}

\bibitem{chencheng18}
Xiuxiong Chen and Jingrui Cheng, \emph{{On the constant scalar curvature
  K\"ahler metrics, general automorphism group}},  (2018), {arXiv:1801.05907}.

\bibitem{chendonaldsonsun15i}
Xiuxiong Chen, Simon Donaldson, and Song Sun, \emph{K\"ahler-{E}instein metrics
  on {F}ano manifolds. {I}: {A}pproximation of metrics with cone
  singularities}, J. Amer. Math. Soc. \textbf{28} (2015), no.~1, 183--197.
  \MR{3264766}

\bibitem{chendonaldsonsun15ii}
\bysame, \emph{K\"ahler-{E}instein metrics on {F}ano manifolds. {II}: {L}imits
  with cone angle less than {$2\pi$}}, J. Amer. Math. Soc. \textbf{28} (2015),
  no.~1, 199--234. \MR{3264767}

\bibitem{chendonaldsonsun15iii}
\bysame, \emph{K\"ahler-{E}instein metrics on {F}ano manifolds. {III}: {L}imits
  as cone angle approaches {$2\pi$} and completion of the main proof}, J. Amer.
  Math. Soc. \textbf{28} (2015), no.~1, 235--278. \MR{3264768}

\bibitem{chensun14}
Xiuxiong Chen and Song Sun, \emph{Calabi flow, geodesic rays, and uniqueness of
  constant scalar curvature {K}\"{a}hler metrics}, Ann. of Math. (2)
  \textbf{180} (2014), no.~2, 407--454. \MR{3224716}

\bibitem{dervansektnan20}
Ruadha\'{\i} Dervan and Lars~Martin Sektnan, \emph{Extremal metrics of
  fibrations}, Proc. Lond. Math. Soc. (3) \textbf{120} (2020), no.~4, 587--616.
  \MR{4008378}

\bibitem{dervansektnan19b}
\bysame, \emph{Moduli theory, stability of fibrations and optimal symplectic
  connections}, Geom. Topol. \textbf{25} (2021), no.~5, 2643--2697.
  \MR{4310897}

\bibitem{dervansektnan19a}
\bysame, \emph{Optimal symplectic connections on holomorphic submersions},
  Comm. Pure Appl. Math. \textbf{74} (2021), no.~10, 2132--2184. \MR{4303016}

\bibitem{dervansektnan21}
\bysame, \emph{Uniqueness of optimal symplectic connections}, Forum Math. Sigma
  \textbf{9} (2021), Paper No. e18, 37. \MR{4228270}

\bibitem{donaldson02}
Simon~K. Donaldson, \emph{Scalar curvature and stability of toric varieties},
  J. Differential Geom. \textbf{62} (2002), no.~2, 289--349. \MR{1988506}

\bibitem{donaldson05}
\bysame, \emph{Lower bounds on the {C}alabi functional}, J. Differential Geom.
  \textbf{70} (2005), no.~3, 453--472. \MR{2192937}

\bibitem{donaldson09}
\bysame, \emph{Constant scalar curvature metrics on toric surfaces}, Geom.
  Funct. Anal. \textbf{19} (2009), no.~1, 83--136. \MR{2507220}

\bibitem{fine04}
Joel Fine, \emph{Constant scalar curvature {K}\"{a}hler metrics on fibred
  complex surfaces}, J. Differential Geom. \textbf{68} (2004), no.~3, 397--432.
  \MR{2144537}

\bibitem{futaki83}
A.~Futaki, \emph{An obstruction to the existence of {E}instein {K}\"{a}hler
  metrics}, Invent. Math. \textbf{73} (1983), no.~3, 437--443. \MR{718940}

\bibitem{hong98}
Ying-Ji Hong, \emph{Ruled manifolds with constant {H}ermitian scalar
  curvature}, Math. Res. Lett. \textbf{5} (1998), no.~5, 657--673. \MR{1666868}

\bibitem{hwangsinger02}
Andrew~D. Hwang and Michael~A. Singer, \emph{A momentum construction for
  circle-invariant {K}\"{a}hler metrics}, Trans. Amer. Math. Soc. \textbf{354}
  (2002), no.~6, 2285--2325. \MR{1885653}

\bibitem{kuranishi65}
M.~Kuranishi, \emph{New proof for the existence of locally complete families of
  complex structures}, Proc. {C}onf. {C}omplex {A}nalysis ({M}inneapolis,
  1964), Springer, Berlin, 1965, pp.~142--154. \MR{0176496}

\bibitem{lebrunsimanca94}
C.~LeBrun and S.~R. Simanca, \emph{Extremal {K}\"ahler metrics and complex
  deformation theory}, Geom. Funct. Anal. \textbf{4} (1994), no.~3, 298--336.
  \MR{1274118}

\bibitem{ortu21}
Annamaria Ortu, \emph{{Deformations of holomorphic submersions and optimal
  symplectic connections}},  (2021), {preliminary draft of forthcoming paper}.

\bibitem{Szethesis}
G{\'a}bor Sz{\'e}kelyhidi, \emph{{E}xtremal {K\"ahler} {M}etrics and
  {K}-stability}, Ph.D. thesis, Imperial {C}ollege {L}ondon, 2006.

\bibitem{szekelyhidi10}
\bysame, \emph{The {K}\"{a}hler-{R}icci flow and {$K$}-polystability}, Amer. J.
  Math. \textbf{132} (2010), no.~4, 1077--1090. \MR{2663648}

\bibitem{szekelyhidi12}
\bysame, \emph{On blowing up extremal {K}\"ahler manifolds}, Duke Math. J.
  \textbf{161} (2012), no.~8, 1411--1453. \MR{2931272}

\bibitem{szekelyhidi14book}
\bysame, \emph{An introduction to extremal {K\"a}hler metrics}, Graduate
  Studies in Mathematics, vol. 152, American Mathematical Society, Providence,
  RI, 2014. \MR{3186384}

\bibitem{tian97}
Gang Tian, \emph{K\"ahler-{E}instein metrics with positive scalar curvature},
  Invent. Math. \textbf{130} (1997), no.~1, 1--37. \MR{1471884}

\bibitem{yau93}
Shing-Tung Yau, \emph{Open problems in geometry}, Differential geometry:
  partial differential equations on manifolds ({L}os {A}ngeles, {CA}, 1990),
  Proc. Sympos. Pure Math., vol.~54, Amer. Math. Soc., Providence, RI, 1993,
  pp.~1--28. \MR{1216573}

\end{thebibliography}
\bibliographystyle{amsplain}

\end{document}